\documentclass[11pt]{article}  % [12pt] option for the benefit of aging markers

\usepackage{amsmath,amsfonts,amssymb,url, amsthm, bm, tensor, empheq, bbm, accents, dsfont,esint,mathtools}   % amssymb package contains more mathematical symbols
\allowdisplaybreaks
\usepackage{graphicx, hyperref, cite}         % graphicx package enables you to paste in graphics
\usepackage{cleveref}
\usepackage{listings}
\usepackage[labelformat=simple]{subcaption}
\usepackage{caption,setspace}
\usepackage[nohead]{geometry}
\usepackage{subcaption}
\usepackage{tabularx}
\usepackage{enumitem}
\usepackage{pdfpages}
\usepackage{epstopdf}
\usepackage{mathrsfs}
\usepackage[T1]{fontenc}
\usepackage[utf8]{inputenc}
\usepackage{titlesec}
\usepackage[nice]{nicefrac}
\usepackage{enumitem}
\usepackage[all]{hypcap}
\usepackage{chngcntr}
	\counterwithin{equation}{section}
\usepackage{blindtext}
\usepackage[bottom]{footmisc}
\usepackage[font=footnotesize,labelfont=bf]{caption}
\usepackage{indentfirst}

\newtheorem{theorem}{Theorem}[]
\newtheorem{lemma}{Lemma}
\newtheorem{remark}{Remark}
\newtheorem{corollary}{Corollary}

\theoremstyle{definition}

\DeclarePairedDelimiter\floor{\lfloor}{\rfloor}

\def\XXint#1#2#3{{\setbox0=\hbox{$#1{#2#3}{\int}$ }
		\vcenter{\hbox{$#2#3$ }}\kern-.6\wd0}}

\title{\bf{On explicit order 1.5  approximations with varying coefficients: the case of super-linear diffusion coefficients}}
\author{Sotirios Sabanis and Ying Zhang \\
\em{School of Mathematics,} \\  \em{The University of
Edinburgh, Edinburgh EH9 3FD, U.K.}}
\begin{document}

\maketitle

\begin{abstract}

\bigskip
A conjecture appears in [Kumar and Sabanis (2016). \url{arXiv:1601.02695[math.PR]}], in the form of a remark, where it is stated that it is possible to construct, in a specified way, any high order explicit numerical schemes to approximate the solutions of SDEs with superlinear coefficients. We answer this conjecture to the positive for the case of order 1.5 approximations and show that the suggested methodology works. Moreover, we explore the case of having H\"{o}lder continuous derivatives for the diffusion coefficients.

\end{abstract}

\section{Introduction}
%A conjecture appears in \cite{milsteinscheme}, in the form of a remark, where it is stated that it is possible to construct, in a specified way, any high order explicit numerical schemes to approximate the solutions of SDEs with superlinear coefficients. We answer this conjecture to the positive for the case of order 1.5 approximations. It is shown, under very mild conditions, that these explicit schemes converge in \(\mathcal{L}^2\) to the solution of the corresponding SDEs.
Due to recent research (see \cite{hutzenthaler2012}, \cite{hutzenthaler2015}, \cite{eulerscheme}, \cite{SabanisAoAP}, \cite{WangGan} and references therein), new explicit Euler-type schemes have been developed to approximate SDEs with superlinearly growing coefficients following the observation in \cite{hutzenthaler2011} that the classical (explicit) Euler scheme cannot be used for such approximations. This has been extended to Milstein-type schemes (see \cite{milsteinscheme}, \cite{Kruse et al.} and references therein). Such schemes are explicit and therefore more computationally efficient compared to the implicit methods. 

In this article, a new type of explicit order 1.5 scheme is constructed. The techniques used in \cite{SabanisAoAP} and \cite{milsteinscheme} are further extended to obtain the \(\mathcal{L}^2\) rate of convergence of the proposed order 1.5 scheme. The main idea is to follow the approach of \cite{Platen-Wagner} by using an appropriate Ito-Taylor (known also as Wagner-Platen) expansion and the taming technique introduced in \cite{SabanisAoAP} and  \cite{milsteinscheme}. Theorem \ref{theorem1} below gives the rate of convergence in \(\mathcal{L}^2\) which is obtained under certain conditions (also given below). In addition, by the combination of the results in \cite{SabanisAoAP}, \cite{milsteinscheme} and in this article, one can arguably anticipate that, by using the uniform taming approach as explained below, any high order (explicit) scheme can be constructed with the desired rate of convergence as in the global Lipschitz case (see \cite{Platen-Wagner}).  

Recent developments in data science attracted our attention to the fact that high order schemes can be used for MCMC algorithms with improved convergence properties in high dimensions. Moreover, such schemes can be combined with multilevel techniques in a natural way One can refer to the article \cite{TULA} on tamed unadjusted Langevin algorithms and consider possible extensions of such techniques to achieve higher accuracy when using tamed schemes to sample from a target distribution, typically the invariant measure of the associated Langevin SDE.

This section is concluded by introducing some notation. The Euclidean norm of a vector \(b \in \mathbb{R}^d\) and the Hilbert-Schmidt norm of a matrix \(\sigma \in \mathbb{R}^{d\times m}\) are denoted by $|b|$ and $|\sigma|$ respectively. $\sigma^{\ast}$ is the transpose matrix of $\sigma$.  The $i$-th element of $b$ and \((i,j)\)-th element of $\sigma$ are denoted respectively by $b^{(i)}$ and \(\sigma^{(i,j)}\), for every \(i = 1, \dots,d\) and \(j = 1,\dots,m\). In addition, \(\floor{a}\) denotes the integer part of a positive real number $a$. The inner product of two vectors \(x,y \in \mathbb{R}^d\) is denoted by $xy$. Let $f:\mathbb{R}^{d} \rightarrow \mathbb{R}$ be a twice continuously differentiable function. Denote by $\nabla f$ and $\nabla^2 f$ the gradient and the Hessian of $f$ respectively. For every \(j = 1, \dots,m\),  define $L^0: C^2(\mathbb{R}^d) \rightarrow C(\mathbb{R}^d)$ and $L^j: C^2(\mathbb{R}^d) \rightarrow C^1(\mathbb{R}^d)$ by
\[
L^0 = \sum_{u=1}^d b^{(u)}\frac{\partial}{\partial x^{(u)}} +\frac{1}{2}\sum_{u,l=1}^d\sum_{j_1 = 1}^m \sigma^{(u,j_1)}\sigma^{(l,j_1)}\frac{\partial^2}{\partial x^{(u)}\partial x^{(l)}}, \quad L^j = \sum_{u=1}^d \sigma^{(u,j)}\frac{\partial}{\partial x^{(u)}}.
\]
Note that, for any $j = 1, \dots, m$, by composing the operator $L^j$ with itself, one obtains $L^jL^{j_1}: C^2(\mathbb{R}^d) \rightarrow C(\mathbb{R}^d)$ for every \(j, j_1 = 1, \dots,m\), which can be written as
\[
L^jL^{j_1} = \sum_{u,l=1}^d \sigma^{(u,j)}\frac{\partial}{\partial x^{(u)}}\sigma^{(l,j_1)}\frac{\partial}{\partial x^{(l)}} +\sum_{u,l=1}^d \sigma^{(u,j)}\sigma^{(l,j_1)}\frac{\partial^2}{\partial x^{(u)}\partial x^{(l)}}.
\]

%%%%%%%%%%%%%%%%%%%%%%%%%%%%%%%%%%%%%%%%%%%%%%%%%%%%%%%%%%%%%%%
\section{Main results}
Let \((\Omega, \{\mathscr{F}_t\}_{t \geq 0}, \mathscr{F},\mathbb{P})\) be a complete filtered probability space satisfying the usual conditions, which means that the filtration is right continuous and \(\mathscr{F}_0\) contains all $\mathbb{P}$-null sets. Denote by \((w_t)_{t\in [0,T]}\) an $m$-dimensional Wiener process. Moreover, assume that $b$ and $\sigma$ are Borel-measurable functions from $\mathbb{R}^d$ to $\mathbb{R}^d$ and $\mathbb{R}^{d\times m}$, respectively. The drift and diffusion coefficients $b$ and $\sigma$ are assumed to be twice continuously differentiable in \(x \in \mathbb{R}^d\). For a fixed \(T>0\), consider a $d$-dimensional SDE,
\begin{equation}\label{sde}
x_t = x_0 + \int_0^t b(x_s)\, ds+\int_0^t \sigma (x_s)\, dw_s,
\end{equation}
almost surely for any \(t \in [0,T]\), where $x_0$ is an $\mathcal{F}_0$-measurable random variable. Let \(p_0 \geq 4\), $p_1>2$, and \(\rho \geq 2\). The following assumptions are stated.
\begin{enumerate}[label=\textbf{A-\arabic*}]
\item \label{a1} \(\mathbb{E}|x_0|^{p_0}<\infty.\)
\item \label{a2}  There exists a constant \(K>0\), such that for any \(x \in \mathbb{R}^d\), %\(\exists K >0\)
\[
2xb(x) +(p_0 -1)|\sigma(x)|^2 \leq K(1+|x|^2).
\]
\item \label{a3} There exists a constant \(K>0\), such that for any \(x, \bar{x} \in \mathbb{R}^d\),
\[
2(x-\bar{x})(b(x)-b(\bar{x})) +(p_1 -1)|\sigma(x)-\sigma (\bar{x})|^2 \leq K|x-\bar{x}|^2.
\]
\item \label{a4} There exists a constant \(K>0\), such that for any \(x,\bar{x} \in \mathbb{R}^d\), and \(i = 1, \dots,d\),
\[
|\nabla^2b^{(i)}(x)-\nabla^2b^{(i)}(\bar{x})| \leq K(1+|x|+|\bar{x}|)^{\rho -2}|x-\bar{x}|.
\]
\item \label{a5} There exist constants \(K>0\) and $\beta \in (0,1]$, such that for any \(x,\bar{x} \in \mathbb{R}^d\), \(i = 1, \dots,d\), and \(j = 1,\dots, m,\)
\[
|\nabla^2\sigma^{(i,j)}(x)-\nabla^2\sigma^{(i,j)}(\bar{x})| \leq K(1+|x|+|\bar{x}|)^{\frac{\rho-4}{2}}|x-\bar{x}|^{\beta}.
\]
\end{enumerate}
\begin{remark}\label{remark1} Assume \ref{a4} and \ref{a5} hold. Then, one can obtain the following estimates in a straightforward manner. In particular, by \ref{a4}, there exists a constant K>0, such that for any \(i, u, l = 1, \dots,d,\) and \( x, \bar{x} \in \mathbb{R}^d,\) 
\[
\left|\frac{\partial^2b^{(i)}(x)}{\partial y^{(u)} \partial y^{(l)}}\right| \leq K(1+|x|)^{\rho -1}.
\]
In addition,
\[
\left|\frac{\partial b^{(i)}(x)}{\partial y^{(u)}}-\frac{\partial b^{(i)}(\bar{x})}{\partial y^{(u)} }\right|\leq K(1+|x|+|\bar{x}|)^{\rho -1}|x-\bar{x}|.
\]
Furthermore, there is a constant \(K>0\) such that for any \(i, u= 1, \dots,d,\) and \( x, \bar{x} \in \mathbb{R}^d,\) 
\[
\left|\frac{\partial b^{(i)}(x)}{\partial y^{(u)}}\right| \leq K(1+|x|)^{\rho},
\]
\[
|b(x)-b(\bar{x})| \leq K(1+|x|+|\bar{x}|)^{\rho}|x-\bar{x}|,
\]
which implies
\[
|b(x)|\leq K(1+|x|)^{\rho+1}.
\]
Similarly, by \ref{a5}, there exists $K>0$, such that for any \(i, u, l = 1, \dots,d,\) \(j = 1, \dots, m\) and \( x \in \mathbb{R}^d,\) 
\[
\left|\frac{\partial^2\sigma^{(i,j)}(x)}{\partial y^{(u)} \partial y^{(l)}}\right| \leq K(1+|x|)^{\frac{\rho-2}{2}},
\]
%where $\left|\frac{\partial^2\sigma^{(i,j)}(x)}{\partial x^u \partial x^l}\right|$ is not constant for all $x \in \mathbb{R}^d$, otherwise the above holds with $\beta =1$. 
Moreover, there exists $K>0$, such that for any \(j = 1, \dots, m\) and \(x, \bar{x} \in \mathbb{R}^d\),
\[
\left|\frac{\partial \sigma^{(i,j)}(x)}{\partial y^{(u)}}-\frac{\partial \sigma^{(i,j)}(\bar{x})}{\partial y^{(u)}}\right|\leq K(1+|x|+|\bar{x}|)^{\frac{\rho-2}{2}}|x-\bar{x}|.
\]
Furthermore, there exists K>0, such that for any \(i, u= 1, \dots,d,\) \(j = 1, \dots, m\) and \( x, \bar{x} \in \mathbb{R}^d,\) 
\[
\left|\frac{\partial \sigma^{(i,j)}(x)}{\partial y^{(u)}}\right| \leq K(1+|x|)^{\frac{\rho}{2}},
\]
\[
|\sigma(x)-\sigma(\bar{x})| \leq K(1+|x|+|\bar{x}|)^{\frac{\rho}{2}}|x-\bar{x}|,
\]
which implies
\[
|\sigma(x)|\leq K(1+|x|)^{\frac{\rho }{2}+1}.
\]
Then, there exists a constant \(K>0\), such that
\[
|L^0b(x)|\leq K(1+|x|)^{2\rho+1}, \quad |L^jb(x)|\leq K(1+|x|)^{\frac{3}{2}\rho+1},
\]
\[
|L^0\sigma(x)|\leq K(1+|x|)^{\frac{3}{2}\rho+1}, \quad |L^j\sigma(x)|\leq K(1+|x|)^{\rho+1},
\]
\[
|L^jL^{j_1}\sigma(x)|\leq K(1+|x|)^{\frac{3}{2}\rho+1}.
\]
\end{remark}
We adopt a uniform taming approach meaning that all terms of interest in the numerical scheme, which are used to approximate the SDE \eqref{sde}, are controlled in the same way, i.e. $\frac{1}{1+n^{-\theta}|x|^{2\rho\theta}}$ is used where $\theta$ represents the desired rate. More concretely, in the order 1.5 paradigm, one constructs, for any \(n \in \mathbb{N}\) and $f \in C^2(\mathbb{R}^d)$,
\[
f^n(x) = \frac{f(x)}{1+n^{-\theta}|x|^{2\rho \theta}}, \quad L^{n,0}f(x):=\frac{L^0f(x)}{1+n^{-\theta}|x|^{2\rho \theta}},
\]
\[
 L^{n,j}f(x):=\frac{L^jf(x)}{1+n^{-\theta}|x|^{2\rho \theta}}, \quad L^{n,j}L^{j_1}f(x):=\frac{L^jL^{j_1} f(x)}{1+n^{-\theta}|x|^{2\rho \theta}},
\]
where \(\theta\) is taken to be $3/2$.
\begin{remark} Throughout this article, the constant $C>0$ may take different values at different places, but it is always independent of \(n \in \mathbb{N}\).
\end{remark}
\begin{remark} \label{remark2} Due to Remark \ref{remark1}, one observes that, there exists a constant $C>0$, such that for any $n \in \mathbb{N}$
\[
|b^n(x) |\leq \min{(C n^{\frac{1}{2}}(1+|x|), |b(x)|}), \quad |\sigma^n(x)  |^2\leq \min{(C n^{\frac{1}{2}}(1+|x|^2), |\sigma(x)|^2}),
\]
\[
|L^{n,0}b(x) |\leq \min{(C n(1+|x|), |L^0b(x)|}), \quad |L^{n,j}b(x) |\leq \min{(C n^{\frac{3}{4}}(1+|x|), |L^jb(x)|}),
\]
\[
|L^{n,0}\sigma(x)|\leq \min{(C n^{\frac{3}{4}}(1+|x|), |L^0\sigma(x)|}), \quad 
|L^{n,j}\sigma(x) |\leq \min{(C n^{\frac{1}{2}}(1+|x|), |L^j\sigma(x)|}),
\]
\[
|L^{n,j}L^{j_1}\sigma(x)|\leq \min{(C n^{\frac{3}{4}}(1+|x|), |L^jL^{j_1} \sigma(x)|}).
\]
\end{remark}
Define \(\kappa(n,t):= \floor{nt}/n\), for any \(t \in [0,T]\). Denote by
\[
b^n_1(t,x) = \int_{\kappa(n,t)}^t L^{n,0}b(x) \, ds, \quad 
b^n_2(t,x)= \sum_j \int_{\kappa(n,t)}^t L^{n,j}b(x) \, dw_s^j,
\]
\[
\tilde{b}^n(t,x) = b^n(x)+b^n_1(t,x) +b^n_2(t,x),
\]
\[
\sigma^n_1(t,x) = \sum_j\int_{\kappa(n,t)}^t L^{n,j}\sigma(x)\, dw_s^j, \quad \sigma^n_2(t,x)  =  \int_{\kappa(n,t)}^t L^{n,0}\sigma(x)\, ds,
\]
\[
\sigma^n_3(t,x) = \sum_j\sum_{j_1}\int_{\kappa(n,t)}^t \int_{\kappa(n,t)}^s L^{n,j}L^{j_1}\sigma(x)\, dw_r^{j}\,dw_s^{j_1},
\]
\[
\tilde{\sigma}^n(t,x) =\sigma^n(x)+ \sigma_M^n(t,x),
\]
where \(\sigma_M^n(t,x) = \sigma^n_1(t,x) +\sigma^n_2(t,x) +\sigma^n_3(t,x)\). The order 1.5 strong Taylor scheme is as follows:
\begin{equation}\label{scheme}
x_t^n = x_0 + \int_0^t \tilde{b}^n(s,x_{\kappa (n,s)}^n) \, ds +\int_0^t \tilde{\sigma}^n(s,x_{\kappa (n,s)}^n)\, dw_s,
\end{equation}
almost surely for any \(t \in [0,T]\).
\begin{theorem}\label{theorem1}
Assume \ref{a1} - \ref{a5} are satisfied with \(p_0 \geq 2(5\rho +1)\), then the explicit order 1.5 scheme \eqref{scheme} converges to the true solution of the SDE \eqref{sde} in \(\mathcal{L}^2\) with a rate of convergence equal to $1+\beta/2$, i.e., there exists a constant $C>0$, such that for any \(n \in \mathbb{N}\),
\begin{equation}\label{l2rate}
\left(\sup_{0\leq t \leq T}\mathbb{E}|x_t - x_t^n|^2\right)^{1/2}\leq Cn^{-(1+\beta/2)}.
\end{equation}
\end{theorem}
Theorem \ref{theorem1} states the convergence result for SDEs with superlinear coefficients. We can also recover the result (Theorem 10.6.3 in \cite{kloeden2011numerical}) for the global Lipschitz case. By the global Lipschitz case, we mean that all the terms in the numerical scheme \eqref{scheme} are assumed to be Lipschitz, i.e. Assumption (6.5) in Theorem 10.6.3 on page 361 from \cite{kloeden2011numerical} is used. Moreover, instead of (6.6) (on page 361 in \cite{kloeden2011numerical}), we assume \ref{a4} and \ref{a5} are satisfied with $\rho=0$, $\beta =1$. \\
Note that by using \ref{a4} and \ref{a5}, the coefficients $b$ and $\sigma$ are only required to be twice continuously differentiable, whereas (6.6) requires higher differentiability from $\sigma$. Note also that due to (6.5) (on page 361 in \cite{kloeden2011numerical}), \ref{a2} and \ref{a3} become redundant. Then, if one examines carefully the proof of Theorem \ref{theorem1}, the following result can be obtained for the global Lipschitz case:
\begin{corollary} \label{corollaryclassical} Let $p_0 > 4$. Assume \ref{a1} and (6.5) (on page 361 in \cite{kloeden2011numerical}), and also assume \ref{a4}, \ref{a5} hold with $\rho =0$, $\beta =1$. Then the explicit order 1.5 scheme converges to the true solution of the SDE in \(\mathcal{L}^2\) with a rate of convergence equal to $1.5$, i.e. there exists a constant $C>0$, such that for any \(n \in \mathbb{N}\)
\begin{equation*}
\left(\sup_{0\leq t \leq T}\mathbb{E}|x_t - x_t^n|^2\right)^{1/2} \leq Cn^{-3/2}.
\end{equation*}
%which corresponds to the classical order 1.5 scheme\footnote{By classical order 1.5 scheme, we mean the case in which drift and diffusion coefficients $b$ and $\sigma$ are Lipschitz functions.}.
\end{corollary}
%%%%%%%%%%%%%%%%%%%%%%%%%%%%%%%%%%%%%%%%%%%%%%%%%%%%%%%%%%%%%%%
\section{Moment bounds}
\begin{lemma}\label{classicallp0} Assume \ref{a1} - \ref{a3} hold. Then, there is a unique solution to the SDE \eqref{sde}, and the \(p_0\)-th moment of the solution is bounded uniformly in time, i.e. there exists a constant $C>0$, such that for any \(t \in [0,T]\),
\[
\sup_{0 \leq t \leq T}\mathbb{E}|x_t|^{p_0} \leq C.
\]
\end{lemma}
\begin{proof}
It is a well-known result, and the proof can be found in \cite{sderesult}.
\end{proof}
\begin{remark} By Remark \ref{remark2}, for each \(n \in \mathbb{N}\), the norm of \(\tilde{b}^n\) and \(\tilde{\sigma}^n\) are growing at most linearly in $x$. %More precisely, there exists a constant \(C>0,\) and a non-negative random variable \(\{M_n\}_{n \geq 1}\) with bounded \(\mathcal{L}^p\) moments, such that
%\[
%|\tilde{b}^n|+|\tilde{\sigma}^n| \leq C(M_n +|x|),
%\]
%almost surely for any \(x \in \mathbb{R}^d\) and \(t \in [0,T]\). One can refer to \cite{schemecoeff} for more detailed explanations.
Then, together with \ref{a1}, this guarantees that for any \(n \in \mathbb{N}\) and \(p \leq p_0\),
\[
\mathbb{E}\left[\sup_{0 \leq t \leq T}|x_t^n|^p\right]< \infty.
\]
\end{remark}
\begin{lemma}\label{lemma2}
Let \ref{a4} - \ref{a5} be satisfied, then there exists a constant $C>0$, such that for any \(n \in \mathbb{N}\) and \(t \in [0,T]\),
\[
\mathbb{E}|b^n_1(t,x_{\kappa(n,t)}^n) |^{p_0} \leq C(1+\mathbb{E}|x_{\kappa(n,t)}^n|^{p_0}),
\]
\[
\mathbb{E}|b^n_2(t,x_{\kappa(n,t)}^n) |^{p_0} \leq Cn^{\frac{p_0}{4}}(1+\mathbb{E}|x_{\kappa(n,t)}^n|^{p_0}),
\]
\[
\mathbb{E}|\sigma^n_1(t,x_{\kappa(n,t)}^n) |^{p_0} \leq C(1+\mathbb{E}|x_{\kappa(n,t)}^n|^{p_0}),
\]
\[
\mathbb{E}|\sigma^n_2(t,x_{\kappa(n,t)}^n) |^{p_0} \leq C(1+\mathbb{E}|x_{\kappa(n,t)}^n|^{p_0}),
\]
\[
\mathbb{E}|\sigma^n_3(t,x_{\kappa(n,t)}^n) |^{p_0} \leq C(1+\mathbb{E}|x_{\kappa(n,t)}^n|^{p_0}).
\]
\end{lemma}
\begin{proof} Due to Remark \ref{remark2}, these inequalities follow immediately.
\end{proof}
\begin{corollary}\label{corollary1}
Assume \ref{a4} - \ref{a5} are satisfied, then there exists a constant $C>0$, such that for any \(n \in \mathbb{N}\) and \(t \in [0,T]\),
\[
\mathbb{E}|\tilde{b}^n(t,x_{\kappa(n,t)}^n) |^{p_0} \leq Cn^{\frac{p_0}{2}}(1+\mathbb{E}|x_{\kappa(n,t)}^n|^{p_0}),
\]
\[
\mathbb{E}|\tilde{\sigma}^n(t,x_{\kappa(n,t)}^n) |^{p_0} \leq Cn^{\frac{p_0}{4}}(1+\mathbb{E}|x_{\kappa(n,t)}^n|^{p_0}).
\]
\end{corollary}
\begin{lemma}\label{lemmalp0}
Assume \ref{a1} - \ref{a5} hold, then there exists a constant $C>0$, such that for any $n \in \mathbb{N}$, the order $1.5$ scheme \eqref{scheme} satisfies
\[
\sup_{n \in \mathbb{N}}\sup_{0\leq t \leq T}\mathbb{E}|x_t^n|^{p_0} \leq C.
\]
\end{lemma}
\begin{proof}
It\^o's formula gives, almost surely,
\begin{align*}
|x_t^n|^{p_0} = \,	&|x_0|^{p_0} +p_0\int_0^t|x_s^n|^{p_0-2}x_s^n \tilde{b}^n(s,x_{\kappa(n,s)}^n) \,ds\\
								&+p_0\int_0^t |x_s^n|^{p_0-2}x_s^n \tilde{\sigma}^n(s,x_{\kappa(n,s)}^n) \, dw_s\\
								&+\frac{p_0}{2}\int_0^t|x_s^n|^{p_0-2}|\tilde{\sigma}^n(s,x_{\kappa(n,s)}^n) |^2 \, ds\\
								&+\frac{p_0(p_0-2)}{2}\int_0^t|x_s^n|^{p_0-4}|\tilde{\sigma}^{n\ast}(s,x_{\kappa(n,s)}^n)x_s^n |^2 \, ds,
\end{align*}
for any \(t \in [0,T]\). Then, since the expectation of the third term above is zero, one obtains
\begin{align*}
\mathbb{E}|x_t^n|^{p_0} \leq \,	& \mathbb{E}|x_0|^{p_0} +p_0\mathbb{E}\int_0^t|x_s^n|^{p_0-2}(x_s^n -x_{\kappa(n,s)}^n)b^n(x_{\kappa(n,s)}^n) \,ds\\
								&+p_0\mathbb{E}\int_0^t|x_s^n|^{p_0-2}x_{\kappa(n,s)}^nb^n(x_{\kappa(n,s)}^n) \,ds\\
								&+p_0\mathbb{E}\int_0^t|x_s^n|^{p_0-2}x_s^n b^n_1(s,x_{\kappa(n,s)}^n) \,ds
								+p_0\mathbb{E}\int_0^t|x_s^n|^{p_0-2}x_s^n b^n_2(s,x_{\kappa(n,s)}^n) \,ds\\
								&+\frac{p_0(p_0-1)}{2}\mathbb{E}\int_0^t|x_s^n|^{p_0-2}|\tilde{\sigma}^{n}(s,x_{\kappa(n,s)}^n)|^2 \, ds,
\end{align*}
which can be written as
\begin{align}
\mathbb{E}|x_t^n|^{p_0} \leq G_1+ \sum_{i=2}^7 G_i(t), \label{lp0}
\end{align}
where $G_1=\mathbb{E}|x_0|^{p_0}$,
\begin{align*}
G_2(t) &= p_0\mathbb{E}\int_0^t|x_s^n|^{p_0-2}(x_s^n -x_{\kappa(n,s)}^n)b^n(x_{\kappa(n,s)}^n) \,ds,\\
G_3(t) &= \frac{p_0}{2}\mathbb{E}\int_0^t|x_s^n|^{p_0-2}(2x_{\kappa(n,s)}^nb^n(x_{\kappa(n,s)}^n) +(p_0 -1)|\sigma^{n}(x_{\kappa(n,s)}^n)|^2)\,ds,\\
G_4(t)&=p_0\mathbb{E}\int_0^t|x_s^n|^{p_0-2}x_s^n b^n_1(s,x_{\kappa(n,s)}^n) \,ds,\\
G_5(t)& =p_0\mathbb{E}\int_0^t|x_s^n|^{p_0-2}x_s^n b^n_2(s,x_{\kappa(n,s)}^n) \,ds,\\
 G_6(t)&=\frac{p_0(p_0-1)}{2}\mathbb{E}\int_0^t|x_s^n|^{p_0-2}|\sigma^{n}_M(s,x_{\kappa(n,s)}^n)|^2 \, ds,\\
G_7(t)& = p_0(p_0-1)\mathbb{E}\int_0^t|x_s^n|^{p_0-2}\sum_{k=1}^d\sum_{v=1}^m\sigma^{n,(k,v)}(x_{\kappa(n,s)}^n)\sigma^{n,(k,v)}_M(s,x_{\kappa(n,s)}^n)\,ds.
\end{align*}
In order to estimate $G_2(t)$, one writes
\begin{align*}
G_2(t) =\,	&p_0\mathbb{E}\int_0^t|x_s^n|^{p_0-2}\int_{\kappa(n,s)}^s\tilde{b}^n(r,x_{\kappa(n,r)}^n)\, dr b^n(x_{\kappa(n,s)}^n) \,ds\\
					&+p_0\mathbb{E}\int_0^t|x_s^n|^{p_0-2}\int_{\kappa(n,s)}^s\tilde{\sigma}^n(r,x_{\kappa(n,r)}^n)\, dw_rb^n(x_{\kappa(n,s)}^n) \,ds,
\end{align*}
for any \(t \in [0,T]\). By applying Young's inequality and Remark \ref{remark2}, the following estimate can be obtained
\begin{align*}
G_2(t) \leq	\,	& C+ C\int_0^t\sup_{0 \leq r \leq s}\mathbb{E}|x_r^n|^{p_0}ds+C\mathbb{E}\int_0^t\left|n^{\frac{1}{2}}\int_{\kappa(n,s)}^s\tilde{b}^n(r,x_{\kappa(n,r)}^n)\, dr \right|^{p_0} \,ds\\
%					&+C\mathbb{E}\int_0^t(1+|x_{\kappa(n,s)}^n|)^{p_0} \,ds\\
					&+p_0\mathbb{E}\int_0^t(|x_s^n|^{p_0-2}-|x_{\kappa(n,s)}^n|^{p_0-2}|)\int_{\kappa(n,s)}^s\tilde{\sigma}^n(r,x_{\kappa(n,r)}^n)\, dw_rb^n(x_{\kappa(n,s)}^n) \,ds\\
					&+p_0\mathbb{E}\int_0^t|x_{\kappa(n,s)}^n|^{p_0-2}\int_{\kappa(n,s)}^s\tilde{\sigma}^n(r,x_{\kappa(n,r)}^n)\, dw_rb^n(x_{\kappa(n,s)}^n) \,ds,
\end{align*}
for any \(t \in [0,T]\). Since the last term above is zero, by taking into consideration the results of Corollary \ref{corollary1} and by applying It\^o's formula, it follows that, almost surely
\begin{align*}
G_2(t) \leq	\,	& C+ C\int_0^t\sup_{0 \leq r \leq s}\mathbb{E}|x_r^n|^{p_0}ds\\
					&+C\mathbb{E}\int_0^t\int_{\kappa(n,s)}^s|x^n_r|^{p_0-4}x_r^n\tilde{b}^n(r,x_{\kappa(n,r)}^n)\, dr\int_{\kappa(n,s)}^s\tilde{\sigma}^n(r,x_{\kappa(n,r)}^n)\, dw_rb^n(x_{\kappa(n,s)}^n) \,ds\\
					&+C\mathbb{E}\int_0^t\int_{\kappa(n,s)}^s|x^n_r|^{p_0-4}x_r^n\tilde{\sigma}^n(r,x_{\kappa(n,r)}^n)\, dw_r\int_{\kappa(n,s)}^s\tilde{\sigma}^n(r,x_{\kappa(n,r)}^n)\, dw_rb^n(x_{\kappa(n,s)}^n) \,ds\\
					&+C\mathbb{E}\int_0^t\int_{\kappa(n,s)}^s|x^n_r|^{p_0-4}|\tilde{\sigma}^n(r,x_{\kappa(n,r)}^n)|^2\, dr|\int_{\kappa(n,s)}^s\tilde{\sigma}^n(r,x_{\kappa(n,r)}^n)\, dw_r||b^n(x_{\kappa(n,s)}^n) |\,ds.
\end{align*}
Due to Remark \ref{remark2},
\begin{align*}
G_2(t) 			& \leq C+ C\int_0^t\sup_{0 \leq r \leq s}\mathbb{E}|x_r^n|^{p_0}ds\\
					&+Cn^{\frac{1}{2}}\mathbb{E}\int_0^t\int_{\kappa(n,s)}^s|x^n_r|^{p_0-3}(1+|x_{\kappa(n,s)}^n|)|\tilde{b}^n(r,x_{\kappa(n,r)}^n)|\, dr
					%& \hspace{5em}  \times
					\left|\int_{\kappa(n,s)}^s 	\tilde{\sigma}^n(r,x_{\kappa(n,r)}^n)\, dw_r\right|\,ds\\
					&+Cn^{\frac{1}{2}}\mathbb{E}\int_0^t\int_{\kappa(n,s)}^s |x^n_r|^{p_0-3}(1+|x_{\kappa(n,s)}^n|)|\tilde{\sigma}^n(r,x_{\kappa(n,r)}^n)|^2\, dr \,ds\\
					&+Cn^{\frac{1}{2}}\mathbb{E}\int_0^t\int_{\kappa(n,s)}^s|x^n_r|^{p_0-4}(1+|x_{\kappa(n,s)}^n|)|\tilde{\sigma}^n(r,x_{\kappa(n,r)}^n)|^2\, dr
					%&\hspace{5em}  \times
					\left|\int_{\kappa(n,s)}^s\tilde{\sigma}^n(r,x_{\kappa(n,r)}^n)\, dw_r\right|\,ds,
\end{align*}
for any \(t \in [0,T]\). Then, the application of Young's inequality yields
\begin{align*}
G_2(t) 			& \leq C+ C\int_0^t\sup_{0 \leq r \leq s}\mathbb{E}|x_r^n|^{p_0}ds\\
					&+C\mathbb{E}\int_0^t n^{\frac{1}{4}}\int_{\kappa(n,s)}^s(1+|x^n_r|^{p_0-2}+|x_{\kappa(n,s)}^n|^{p_0-2})|\tilde{b}^n(r,x_{\kappa(n,r)}^n)|\, dr\\
					& \hspace{5em}  \times n^{\frac{1}{4}}\left|\int_{\kappa(n,s)}^s 	\tilde{\sigma}^n(r,x_{\kappa(n,r)}^n)\, dw_r\right|\,ds\\
					 &+C\mathbb{E}\int_0^t\int_{\kappa(n,s)}^sn^{1-\frac{2}{p_0}}(1+|x^n_r|^{p_0-2}+|x_{\kappa(n,s)}^n|^{p_0-2})n^{-\frac{1}{2}+\frac{2}{p_0}}|\tilde{\sigma}^n(r,x_{\kappa(n,r)}^n)|^2\, dr \,ds\\
					&+C\mathbb{E}\int_0^tn^{\frac{1}{4}}\int_{\kappa(n,s)}^s(1+|x^n_r|^{p_0-3}+|x_{\kappa(n,s)}^n|^{p_0-3})|\tilde{\sigma}^n(r,x_{\kappa(n,r)}^n)|^2\, dr\\
					&\hspace{5em}  \times n^{\frac{1}{4}}\left|\int_{\kappa(n,s)}^s\tilde{\sigma}^n(r,x_{\kappa(n,r)}^n)\, dw_r\right|\,ds,
\end{align*}
which can be further estimated as
\begin{align*}
G_2(t) 			& \leq C+ C\int_0^t\sup_{0 \leq r \leq s}\mathbb{E}|x_r^n|^{p_0}ds\\
				&+C\mathbb{E}\int_0^t \left(\int_{\kappa(n,s)}^sn^{\frac{3}{4}-\frac{1}{p_0}}(1+|x^n_r|^{p_0-2}+|x_{\kappa(n,s)}^n|^{p_0-2})n^{-\frac{1}{2}+\frac{1}{p_0}}|\tilde{b}^n(r,x_{\kappa(n,r)}^n)|\, dr\right)^{\frac{p_0}{p_0-1}}\,ds\\
				&+Cn\mathbb{E}\int_0^t\int_{\kappa(n,s)}^s(1+|x^n_r|^{p_0}+|x_{\kappa(n,s)}^n|^{p_0})\,dr\,ds \\
				&+ Cn^{-\frac{p_0}{4}+1}\mathbb{E}\int_0^t\int_{\kappa(n,s)}^s|\tilde{\sigma}^n(r,x_{\kappa(n,r)}^n)|^{p_0}\,dr\,ds\\
				&+C\mathbb{E}\int_0^t \left(\int_{\kappa(n,s)}^sn^{\frac{3}{4}-\frac{2}{p_0}}(1+|x^n_r|^{p_0-3}+|x_{\kappa(n,s)}^n|^{p_0-3})n^{-\frac{1}{2}+\frac{2}{p_0}}|\tilde{\sigma}^n(r,x_{\kappa(n,r)}^n)|^2\, dr\right)^{\frac{p_0}{p_0-1}}\,ds\\
				&+Cn^{\frac{p_0}{4}}\int_0^t \mathbb{E}\left|\int_{\kappa(n,s)}^s 	\tilde{\sigma}^n(r,x_{\kappa(n,r)}^n)\, dw_r\right|^{p_0}\,ds
\end{align*}
for any \(t \in [0,T]\). By using Young's inequality and Corollary \ref{corollary1}, one obtains
\begin{align*}
G_2(t) 			& \leq C+ C\int_0^t\sup_{0 \leq r \leq s}\mathbb{E}|x_r^n|^{p_0}ds\\
				&+C\mathbb{E}\int_0^t \left(\int_{\kappa(n,s)}^sn^{\frac{3p_0-4}{4p_0}\times \frac{p_0-1}{p_0-2}}(1+|x^n_r|^{p_0-1}+|x_{\kappa(n,s)}^n|^{p_0-1})\, dr\right)^{\frac{p_0}{p_0-1}}\,ds\\
				&+C\mathbb{E}\int_0^t \left(\int_{\kappa(n,s)}^sn^{\frac{(2-p_0)\times (p_0-1)}{2p_0}}|\tilde{b}^n(r,x_{\kappa(n,r)}^n)|^{p_0-1}\, dr\right)^{\frac{p_0}{p_0-1}}\,ds\\
				&+C\mathbb{E}\int_0^t \left(\int_{\kappa(n,s)}^sn^{\frac{3p_0-8}{4p_0}\times \frac{p_0-1}{p_0-3}}(1+|x^n_r|^{p_0-1}+|x_{\kappa(n,s)}^n|^{p_0-1})\, dr\right)^{\frac{p_0}{p_0-1}}\,ds\\
				&+C\mathbb{E}\int_0^t \left(\int_{\kappa(n,s)}^sn^{\frac{4-p_0}{2p_0}\times\frac{p_0-1}{2}}|\tilde{\sigma}^n(r,x_{\kappa(n,r)}^n)|^{p_0-1}\, dr\right)^{\frac{p_0}{p_0-1}}\,ds\\
				&+Cn^{-\frac{p_0}{4}+1}\int_0^t \int_{\kappa(n,s)}^s \mathbb{E}|\tilde{\sigma}^n(r,x_{\kappa(n,r)}^n)|^{p_0}\, dr\,ds,
\end{align*}
which, due to H\"{o}lder's inequality and Corollary \ref{corollary1}, implies
\begin{align*}
G_2(t) 			& \leq C+ C\int_0^t\sup_{0 \leq r \leq s}\mathbb{E}|x_r^n|^{p_0}ds\\
				&+Cn^{\frac{3p_0-4}{4(p_0-2)}-\frac{1}{p_0-1}}\int_0^t \mathbb{E}\int_{\kappa(n,s)}^s (1+|x^n_r|^{p_0}+|x_{\kappa(n,s)}^n|^{p_0})\, dr\,ds\\
				&+Cn^{-\frac{p_0}{2}+1-\frac{1}{p_0-1}}\int_0^t \mathbb{E}\int_{\kappa(n,s)}^s|\tilde{b}^n(r,x_{\kappa(n,r)}^n)|^{p_0}\, dr\,ds\\
				&+Cn^{\frac{3p_0-8}{4(p_0-3)}-\frac{1}{p_0-1}}\int_0^t \mathbb{E}\int_{\kappa(n,s)}^s(1+|x^n_r|^{p_0}+|x_{\kappa(n,s)}^n|^{p_0})\, dr\,ds,
	%			&+Cn^{-\frac{p_0}{4}+1-\frac{1}{p_0-1}}\int_0^t \mathbb{E}\int_{\kappa(n,s)}^s|\tilde{\sigma}^n(r,x_{\kappa(n,r)}^n)|^{p_0}\, dr\,ds,
\end{align*}
for any \(t \in [0,T]\). Note that in the third and fifth term above, \(n^{\frac{3p_0-4}{4(p_0-2)}}\) and \(n^{\frac{3p_0-8}{4(p_0-3)}}\) are less than $n$ for all \(p_0 \geq 4\). Thus, in view of Corollary \ref{corollary1}, one obtains
\[
G_2(t) 	 \leq C+ C\int_0^t\sup_{0 \leq r \leq s}\mathbb{E}|x_r^n|^{p_0}ds,
\]
for any \(t \in [0,T]\). For $G_3(t)$, applying \ref{a2} gives
\begin{align*}
G_3(t) 		& = \frac{p_0}{2}\mathbb{E}\int_0^t|x_s^n|^{p_0-2}\frac{2x_{\kappa(n,s)}^nb(x_{\kappa(n,s)}^n) +(p_0 -1)|\sigma(x_{\kappa(n,s)}^n)|^2}{1+n^{-3/2}|x_{\kappa(n,s)}^n|^{3\rho}}\,ds\\
			&\leq C\mathbb{E}\int_0^t|x_s^n|^{p_0-2}(1+|x_{\kappa(n,s)}^n|^2)\,ds,
\end{align*}
which, due to Young's inequality, results in
\[
G_3(t) 	 \leq C+ C\int_0^t\sup_{0 \leq r \leq s}\mathbb{E}|x_r^n|^{p_0}ds,
\]
for any \(t \in [0,T]\). To estimate $G_4(t)$, one uses Young's inequality to obtain
\[
G_4(t)	\leq C\mathbb{E}\int_0^t|x_s^n|^{p_0}\,ds +C\mathbb{E}\int_0^t|b^n_1(s,x_{\kappa(n,s)}^n) |^{p_0}\,ds,
\]
which implies due to Lemma \ref{lemma2},
\[
G_4(t) 	 \leq C+ C\int_0^t\sup_{0 \leq r \leq s}\mathbb{E}|x_r^n|^{p_0}ds,
\]
for any \(t \in [0,T]\). Moreover, one writes
\[
G_5(t) =\sum_{i=1}^3 G_{5i}(t),
\]
where
\begin{align*}
G_{51}(t)&=p_0\mathbb{E}\int_0^t|x_s^n|^{p_0-2}(x_s^n-x_{\kappa(n,s)}^n) b^n_2(s,x_{\kappa(n,s)}^n) \,ds,\\
G_{52}(t)&=p_0\mathbb{E}\int_0^t(|x_s^n|^{p_0-2}-|x_{\kappa(n,s)}^n|^{p_0-2})x_{\kappa(n,s)}^n b^n_2(s,x_{\kappa(n,s)}^n) \,ds, \\
G_{53}(t)		&=p_0\mathbb{E}\int_0^t|x_{\kappa(n,s)}^n|^{p_0-2}x_{\kappa(n,s)}^n b^n_2(s,x_{\kappa(n,s)}^n) \,ds.
\end{align*}
One then calculates the following
\begin{align*}
G_{51}(t) =\,		&p_0\mathbb{E}\int_0^t|x_s^n|^{p_0-2}\int_{\kappa(n,s)}^s\tilde{b}^n(r,x_{\kappa(n,r)}^n)\,dr b^n_2(s,x_{\kappa(n,s)}^n) \,ds\\
					&+p_0\mathbb{E}\int_0^t|x_s^n|^{p_0-2}\int_{\kappa(n,s)}^s\tilde{\sigma}^n(r,x_{\kappa(n,r)}^n)\,dw_r b^n_2(s,x_{\kappa(n,s)}^n) \,ds,
\end{align*}
which implies, due to Young's inequality,
\begin{align*}
G_{51}(t)	\leq\,	&C\mathbb{E}\int_0^t|x_s^n|^{p_0}\,ds +C\mathbb{E}\int_0^t\left|n^{\frac{1}{4}}\int_{\kappa(n,s)}^s\tilde{b}^n(r,x_{\kappa(n,r)}^n)\,dr n^{-\frac{1}{4}}b^n_2(s,x_{\kappa(n,s)}^n)\right|^{\frac{p_0}{2}}\,ds\\
						&+C\mathbb{E}\int_0^t\left|n^{\frac{1}{4}}\int_{\kappa(n,s)}^s\tilde{\sigma}^n(r,x_{\kappa(n,r)}^n)\,dw_rn^{-\frac{1}{4}} b^n_2(s,x_{\kappa(n,s)}^n) \right|^{\frac{p_0}{2}}\,ds,
\end{align*}
for any \(t \in [0,T]\). Then, on applying Young's inequality again, one obtains
\begin{align*}
G_{51}(t)	\leq\,	&C\mathbb{E}\int_0^t|x_s^n|^{p_0}\,ds +Cn^{\frac{p_0}{4}}\mathbb{E}\int_0^t\left|\int_{\kappa(n,s)}^s\tilde{b}^n(r,x_{\kappa(n,r)}^n)\,dr \right|^{p_0}\,ds \\
						&+Cn^{\frac{p_0}{4}} \mathbb{E}\int_0^t\left|\int_{\kappa(n,s)}^s\tilde{\sigma}^n(r,x_{\kappa(n,r)}^n)\,dw_r \right|^{p_0}\,ds+Cn^{-\frac{p_0}{4}}\mathbb{E}\int_0^t|b^n_2(s,x_{\kappa(n,s)}^n)|^{p_0}\,ds,
\end{align*}
which by using H\"{o}lder's inequality and Lemma \ref{lemma2} yields
\begin{align*}
G_{51}(t)	\leq\,	&C+ C\int_0^t\sup_{0 \leq r \leq s}\mathbb{E}|x_r^n|^{p_0}ds\\
						&+Cn^{\frac{p_0}{4}-p_0+1}\int_0^t\int_{\kappa(n,s)}^s\mathbb{E}|\tilde{b}^n(r,x_{\kappa(n,r)}^n)|^{p_0}\,dr\,ds \\
						&+Cn^{\frac{p_0}{4}-\frac{p_0}{2}+1} \int_0^t\int_{\kappa(n,s)}^s\mathbb{E}|\tilde{\sigma}^n(r,x_{\kappa(n,r)}^n)|^{p_0}\,dr\,ds,
\end{align*}
for any \(t \in [0,T]\). Due to Corollary \ref{corollary1}, one concludes that
\begin{equation}\label{g51}
G_{51}(t) 	 \leq C+ C\int_0^t\sup_{0 \leq r \leq s}\mathbb{E}|x_r^n|^{p_0}ds,
\end{equation}
for any \(t \in [0,T]\). As for $G_{52}(t)$, It\^o's formula gives, almost surely
\begin{align*}
G_{52}(t) \leq \,		& C\mathbb{E}\int_0^t\int_{\kappa(n,s)}^s|x_r^n|^{p_0-4}x_r^n\tilde{b}^n(r,x_{\kappa(n,r)}^n)\,drx_{\kappa(n,s)}^nb_2^n(s,x_{\kappa(n,s)}^n)\,ds \nonumber\\
					&+C\mathbb{E}\int_0^t\int_{\kappa(n,s)}^s|x_r^n|^{p_0-4}x_r^n\tilde{\sigma}^n(r,x_{\kappa(n,r)}^n)\,dw_rx_{\kappa(n,s)}^n \sum_j \int_{\kappa(n,s)}^s L^{n,j}b(x_{\kappa(n,r)}^n) \, dw_r^j\,ds \nonumber\\
					&+C\mathbb{E}\int_0^t\int_{\kappa(n,s)}^s|x_r^n|^{p_0-4}|\tilde{\sigma}^n(r,x_{\kappa(n,r)}^n)|^2\,dr|x_{\kappa(n,s)}^n||b_2^n(s,x_{\kappa(n,s)}^n)|\,ds, \nonumber
\end{align*}
which, by Young's inequality, can be expressed as
\begin{align*}
G_{52}(t) \leq \,	& C\int_0^t\mathbb{E}\int_{\kappa(n,s)}^sn^{\frac{3}{4}-\frac{1}{p_0}}(1+|x_r^n|^{p_0-2}+|x_{\kappa(n,s)}^n|^{p_0-2})n^{-\frac{1}{2}+\frac{1}{p_0}}|\tilde{b}^n(r,x_{\kappa(n,r)}^n)|\,dr\\
						&\hspace{5em}  \times n^{-\frac{1}{4}}|b_2^n(s,x_{\kappa(n,s)}^n)|\,ds \nonumber\\
						&+C\sum_{j=1}^m\int_0^t\mathbb{E}\int_{\kappa(n,s)}^s(1+|x_r^n|^{p_0-2}+|x_{\kappa(n,s)}^n|^{p_0-2})|\tilde{\sigma}^n(r,x_{\kappa(n,r)}^n)||L^{n,j}b(x_{\kappa(n,r)}^n)| \,dr\,ds \nonumber\\
						 &+C\int_0^t\mathbb{E}\int_{\kappa(n,s)}^sn^{\frac{3}{4}-\frac{2}{p_0}}(1+|x_r^n|^{p_0-3}+|x_{\kappa(n,s)}^n|^{p_0-3})n^{-\frac{1}{2}+\frac{2}{p_0}}|\tilde{\sigma}^n(r,x_{\kappa(n,r)}^n)|^2\,dr\\
						&\hspace{5em}  \times n^{-\frac{1}{4}}|b_2^n(s,x_{\kappa(n,s)}^n)|\,ds, \nonumber
\end{align*}
for any \(t \in [0,T]\). One uses Young's inequality again and Remark \ref{remark2} to obtain
\begin{align*}
G_{52}(t) \leq \,	& C\int_0^t\mathbb{E}\left(\int_{\kappa(n,s)}^sn^{\frac{3}{4}-\frac{1}{p_0}}(1+|x_r^n|^{p_0-2}+|x_{\kappa(n,s)}^n|^{p_0-2})n^{-\frac{1}{2}+\frac{1}{p_0}}|\tilde{b}^n(r,x_{\kappa(n,r)}^n)|\,dr\right)^{\frac{p_0}{p_0-1}}\,ds\\
						 &+C\int_0^t\mathbb{E}\int_{\kappa(n,s)}^sn^{1-\frac{1}{p_0}}(1+|x_r^n|^{p_0-1}+|x_{\kappa(n,s)}^n|^{p_0-1})n^{-\frac{1}{4}+\frac{1}{p_0}}|\tilde{\sigma}^n(r,x_{\kappa(n,r)}^n)|\,dr\,ds \nonumber\\
						 &+C\int_0^t\mathbb{E}\left(\int_{\kappa(n,s)}^sn^{\frac{3}{4}-\frac{2}{p_0}}(1+|x_r^n|^{p_0-3}+|x_{\kappa(n,s)}^n|^{p_0-3})n^{-\frac{1}{2}+\frac{2}{p_0}}|\tilde{\sigma}^n(r,x_{\kappa(n,r)}^n)|^2\,dr\right)^{\frac{p_0}{p_0-1}}\,ds\\
						&+Cn^{-\frac{p_0}{4}}\int_0^t\mathbb{E}|b_2^n(s,x_{\kappa(n,s)}^n)|^{p_0}\,ds,
\end{align*}
which implies due to Lemma \ref{lemma2}
\begin{align*}
G_{52}(t) \leq \,	& C\int_0^t\mathbb{E}\left(\int_{\kappa(n,s)}^sn^{\frac{3p_0-4}{4p_0}\times\frac{p_0-1}{p_0-2}}(1+|x_r^n|^{p_0-1}+|x_{\kappa(n,s)}^n|^{p_0-1})\,dr\right)^{\frac{p_0}{p_0-1}}\,ds\\
						&+C\int_0^t\mathbb{E}\left(\int_{\kappa(n,s)}^sn^{\frac{(2-p_0)\times(p_0-1)}{2p_0}}|\tilde{b}^n(r,x_{\kappa(n,r)}^n)|^{p_0-1}\,dr\right)^{\frac{p_0}{p_0-1}}\,ds\\
						&+C\int_0^t\mathbb{E}\int_{\kappa(n,s)}^sn(1+|x_r^n|^{p_0}+|x_{\kappa(n,s)}^n|^{p_0})\,dr\,ds\\
						&+C\int_0^t\mathbb{E}\int_{\kappa(n,s)}^sn^{\frac{4-p_0}{4p_0}\times p_0}|\tilde{\sigma}^n(r,x_{\kappa(n,r)}^n)|^{p_0}\,dr\,ds \nonumber\\
						 &+C\int_0^t\mathbb{E}\left(\int_{\kappa(n,s)}^sn^{\frac{3p_0-8}{4p_0}\times\frac{p_0-1}{p_0-3}}(1+|x_r^n|^{p_0-1}+|x_{\kappa(n,s)}^n|^{p_0-1})\,dr\right)^{\frac{p_0}{p_0-1}}\,ds\\
						&+C\int_0^t\mathbb{E}\left(\int_{\kappa(n,s)}^sn^{\frac{4-p_0}{2p_0}\times\frac{p_0-1}{2}}|\tilde{\sigma}^n(r,x_{\kappa(n,r)}^n)|^{p_0-1}\,dr\right)^{\frac{p_0}{p_0-1}}\,ds\\
						&+C+ C\int_0^t\sup_{0 \leq r \leq s}\mathbb{E}|x_r^n|^{p_0}ds,
\end{align*}
for any \(t \in [0,T]\). By using H\"{o}lder's inequality and Corollary \ref{corollary1},
\begin{align*}
G_{52}(t) \leq \,	& Cn^{\frac{3p_0-4}{4(p_0-2)}-\frac{1}{p_0-1}}\int_0^t\mathbb{E}\int_{\kappa(n,s)}^s(1+|x_r^n|^{p_0}+|x_{\kappa(n,s)}^n|^{p_0})\,dr\,ds\\
						&+Cn^{-\frac{p_0}{2}+1-\frac{1}{p_0-1}}\int_0^t\mathbb{E}\int_{\kappa(n,s)}^s|\tilde{b}^n(r,x_{\kappa(n,r)}^n)|^{p_0}\,dr\,ds\\
						&+Cn^{\frac{3p_0-8}{4(p_0-3)}-\frac{1}{p_0-1}}\int_0^t\mathbb{E}\int_{\kappa(n,s)}^s(1+|x_r^n|^{p_0}+|x_{\kappa(n,s)}^n|^{p_0})\,dr\,ds\\
						&+Cn^{-\frac{p_0}{4}+1-\frac{1}{p_0-1}}\int_0^t\mathbb{E}\int_{\kappa(n,s)}^s|\tilde{\sigma}^n(r,x_{\kappa(n,r)}^n)|^{p_0}\,dr\,ds\\
						&+C+ C\int_0^t\sup_{0 \leq r \leq s}\mathbb{E}|x_r^n|^{p_0}ds,
\end{align*}
for any \(t \in [0,T]\). One observes that \(n^{\frac{3p_0-4}{4(p_0-2)}}\) and \(n^{\frac{3p_0-8}{4(p_0-3)}}\) are less than $n$ for all \(p_0 \geq 4,\) then due to Corollary \ref{corollary1}, the following holds
\begin{equation}\label{g52}
G_{52}(t) 	 \leq C+ C\int_0^t\sup_{0 \leq r \leq s}\mathbb{E}|x_r^n|^{p_0}ds,
\end{equation}
for any \(t \in [0,T]\). In addition, note that by the definition of $b^n_2(t,x)$, one obtains
\begin{equation}\label{g53}
G_{53}(t):=p_0\mathbb{E}\int_0^t|x_{\kappa(n,s)}^n|^{p_0-2}x_{\kappa(n,s)}^n b^n_2(s,x_{\kappa(n,s)}^n) \,ds  = 0,
\end{equation}
for any \(t \in [0,T]\). Then, substituting \eqref{g51}, \eqref{g52} and \eqref{g53} into \eqref{g5t}, one obtains
\begin{equation}\label{g5t}
G_5(t) 	 \leq C+ C\int_0^t\sup_{0 \leq r \leq s}\mathbb{E}|x_r^n|^{p_0}ds,
\end{equation}
for any \(t \in [0,T]\). In order to estimate $G_6(t)$, one applies Young's inequality to obtain
\begin{align*}
G_6(t) \leq \,	& C\mathbb{E}\int_0^t|x_s^n|^{p_0}\,ds +C\mathbb{E}\int_0^t|\sigma^{n}_M(s,x_{\kappa(n,s)}^n)|^{p_0}\,ds \\
		\leq \,	& C\mathbb{E}\int_0^t|x_s^n|^{p_0}\,ds +C\mathbb{E}\int_0^t|\sigma^{n}_1(s,x_{\kappa(n,s)}^n)|^{p_0}\,ds \\
					&+C\mathbb{E}\int_0^t|\sigma^{n}_2(s,x_{\kappa(n,s)}^n)|^{p_0}\,ds +C\mathbb{E}\int_0^t|\sigma^{n}_3(s,x_{\kappa(n,s)}^n)|^{p_0}\,ds,
\end{align*}
which implies due to Lemma \ref{lemma2}
\[
G_6(t) 	 \leq C+ C\int_0^t\sup_{0 \leq r \leq s}\mathbb{E}|x_r^n|^{p_0}ds,
\]
for any \(t \in [0,T]\). Finally, for $G_7(t)$, one writes
\begin{equation}\label{g7}
G_7(t) =\sum_{i=1}^2 G_{7i}(t),
\end{equation}
where
\begin{align*}
G_{71}(t)&=p_0(p_0-1)\mathbb{E}\int_0^t(|x_s^n|^{p_0-2}-|x_{\kappa(n,s)}^n|^{p_0-2})\sum_{k=1}^d\sum_{v=1}^m\sigma^{n,(k,v)}(x_{\kappa(n,s)}^n)\sigma^{n,(k,v)}_M(s,x_{\kappa(n,s)}^n)  \, ds, \\
G_{72}(t)&=p_0(p_0-1)\mathbb{E}\int_0^t|x_{\kappa(n,s)}^n|^{p_0-2}\sum_{k=1}^d\sum_{v=1}^m\sigma^{n,(k,v)}(x_{\kappa(n,s)}^n)\sigma^{n,(k,v)}_M(s,x_{\kappa(n,s)}^n)  \, ds.
\end{align*}
To estimate $G_{71}(t)$, It\^o's formula gives, almost surely
\begin{align*}
G_{71}(t) := \,		&C\mathbb{E}\int_0^t(|x_s^n|^{p_0-2}-|x_{\kappa(n,s)}^n|^{p_0-2})\sum_{k=1}^d\sum_{v=1}^m\sigma^{n,(k,v)}(x_{\kappa(n,s)}^n)\sigma^{n,(k,v)}_M(s,x_{\kappa(n,s)}^n)  \, ds\\
			\leq \,	 &C\mathbb{E}\int_0^t\int_{\kappa(n,s)}^s|x_r^n|^{p_0-4}x_r^n\tilde{b}^n(r,x_{\kappa(n,r)}^n)\,dr\sum_{k=1}^d\sum_{v=1}^m\sigma^{n,(k,v)}(x_{\kappa(n,s)}^n)\sigma^{n,(k,v)}_M(s,x_{\kappa(n,s)}^n) \, ds\\
						 &+C\mathbb{E}\int_0^t\int_{\kappa(n,s)}^s|x_r^n|^{p_0-4}x_r^n\tilde{\sigma}^n(r,x_{\kappa(n,r)}^n)\,dw_r\sum_{k=1}^d\sum_{v=1}^m\sigma^{n,(k,v)}(x_{\kappa(n,s)}^n)\sigma^{n,(k,v)}_M(s,x_{\kappa(n,s)}^n) \, ds\\
						 &+C\mathbb{E}\int_0^t\int_{\kappa(n,s)}^s|x_r^n|^{p_0-4}|\tilde{\sigma}^n(r,x_{\kappa(n,r)}^n)|^2\,dr|\sum_{k=1}^d\sum_{v=1}^m\sigma^{n,(k,v)}(x_{\kappa(n,s)}^n)\sigma^{n,(k,v)}_M(s,x_{\kappa(n,s)}^n) |\, ds,
\end{align*}
which by using Remark \ref{remark2} implies
\begin{align*}
G_{71}(t) \leq \,	&Cn^{\frac{1}{4}}\mathbb{E}\int_0^t\int_{\kappa(n,s)}^s|x_r^n|^{p_0-3}(1+|x_{\kappa(n,s)}^n|)|\tilde{b}^n(r,x_{\kappa(n,r)}^n)|\,dr|\sigma^n_M(s,x_{\kappa(n,s)}^n)|\, ds\\
						&+C\mathbb{E}\int_0^t\int_{\kappa(n,s)}^s|x_r^n|^{p_0-4}x_r^n\tilde{\sigma}^n(r,x_{\kappa(n,r)}^n)\,dw_r\\
						& \hspace{3em}  \times \sum_{k=1}^d\sum_{v=1}^m\sigma^{n,(k,v)}(x_{\kappa(n,s)}^n)\sum_{j=1}^m\int_{\kappa(n,s)}^s L^{n,j}\sigma^{(k,v)}(x_{\kappa(n,r)}^n)\, dw_r^j\, ds\\
						&+C\mathbb{E}\int_0^t\int_{\kappa(n,s)}^s|x_r^n|^{p_0-4}x_r^n\tilde{\sigma}^n(r,x_{\kappa(n,r)}^n)\,dw_r\\
						& \hspace{3em}  \times \sum_{k=1}^d\sum_{v=1}^m\sigma^{n,(k,v)}(x_{\kappa(n,s)}^n)\int_{\kappa(n,s)}^s L^{n,0}\sigma^{(k,v)}(x_{\kappa(n,r)}^n)\, dr\, ds\\
						&+C\mathbb{E}\int_0^t\int_{\kappa(n,s)}^s|x_r^n|^{p_0-4}x_r^n\tilde{\sigma}^n(r,x_{\kappa(n,r)}^n)\,dw_r\\
						& \hspace{3em}  \times \sum_{k=1}^d\sum_{v=1}^m\sigma^{n,(k,v)}(x_{\kappa(n,s)}^n)\sum_{j=1}^m\sum_{j_1=1}^m\int_{\kappa(n,s)}^s \int_{\kappa(n,r)}^r L^{n,j}L^{j_1}\sigma^{(k,v)}(x_{\kappa(n,\gamma)}^n)\,dw_{\gamma}^{j_1}\, dw_r^j\, ds\\
						 &+Cn^{\frac{1}{4}}\mathbb{E}\int_0^t\int_{\kappa(n,s)}^s|x_r^n|^{p_0-4}(1+|x_{\kappa(n,s)}^n|)|\tilde{\sigma}^n(r,x_{\kappa(n,r)}^n)|^2\,dr|\sigma^{n}_M(s,x_{\kappa(n,s)}^n)|\, ds,
\end{align*}
for any \(t \in [0,T]\). One then observes that, since \(L^{n,0}\sigma(x_{\kappa(n,r)}^n)\) takes the same value for all \(r \in [\kappa(n,s),s]\), it can be taken out of the integral in the third term above, and thus the third term is zero. Moreover, by Young's inequality and Remark \ref{remark2}, one obtains
\begin{align*}
G_{71}(t) \leq \,	&C\mathbb{E}\int_0^t\int_{\kappa(n,s)}^sn^{\frac{1}{4}}(1+|x_r^n|^{p_0-2}+|x_{\kappa(n,s)}^n|^{p_0-2})|\tilde{b}^n(r,x_{\kappa(n,r)}^n)|\,dr|\sigma^n_M(s,x_{\kappa(n,s)}^n)|\, ds\\
						&+C\mathbb{E}\int_0^t\int_{\kappa(n,s)}^sn^{\frac{3}{4}}|x_r^n|^{p_0-3}(1+|x_{\kappa(n,s)}^n|)^2|\tilde{\sigma}^n(r,x_{\kappa(n,r)}^n)|\, dr\, ds\\
						 &+C\sum_{j=1}^m\mathbb{E}\int_0^t\int_{\kappa(n,s)}^sn^{\frac{3}{4}-\frac{2}{p_0}}|x_r^n|^{p_0-3}(1+|x_{\kappa(n,s)}^n|)n^{-\frac{1}{4}+\frac{1}{p_0}}|\tilde{\sigma}^n(r,x_{\kappa(n,r)}^n)|\\
						& \hspace{5em}  \times n^{-\frac{1}{4}+\frac{1}{p_0}}\left|\sum_{j_1=1}^d\int_{\kappa(n,r)}^r L^{n,j}L^{j_1}\sigma(x_{\kappa(n,\gamma)}^n)\,dw_{\gamma}^{j_1}\right|\, dr\, ds\\
						 &+C\mathbb{E}\int_0^t\int_{\kappa(n,s)}^sn^{\frac{1}{4}}(1+|x_r^n|^{p_0-3}+|x_{\kappa(n,s)}^n|^{p_0-3})|\tilde{\sigma}^n(r,x_{\kappa(n,r)}^n)|^2\,dr|\sigma^{n,(i,j)}_M(s,x_{\kappa(n,s)}^n)|\, ds,
\end{align*}
which yields, due to Young's inequality,
\begin{align*}
G_{71}(t) \leq \,	 &C\int_0^t\mathbb{E}\left(\int_{\kappa(n,s)}^sn^{\frac{3}{4}-\frac{1}{p_0}}(1+|x_r^n|^{p_0-2}+|x_{\kappa(n,s)}^n|^{p_0-2})n^{-\frac{1}{2}+\frac{1}{p_0}}|\tilde{b}^n(r,x_{\kappa(n,r)}^n)|\,dr\right)^{\frac{p_0}{p_0-1}}\,ds\\
						 &+C\mathbb{E}\int_0^t\int_{\kappa(n,s)}^sn^{1-\frac{1}{p_0}}(1+|x_r^n|^{p_0-1}+|x_{\kappa(n,s)}^n|^{p_0-1})n^{-\frac{1}{4}+\frac{1}{p_0}}|\tilde{\sigma}^n(r,x_{\kappa(n,r)}^n)|\, dr\, ds\\
						 &+C\mathbb{E}\int_0^t\int_{\kappa(n,s)}^s\left(n^{\frac{3}{4}-\frac{2}{p_0}}(1+|x_r^n|^{p_0-2}+|x_{\kappa(n,s)}^n|^{p_0-2})n^{-\frac{1}{4}+\frac{1}{p_0}}|\tilde{\sigma}^n(r,x_{\kappa(n,r)}^n)|\right)^{\frac{p_0}{p_0-1}}\,dr\,ds\\
						 &+C\int_0^t\mathbb{E}\left(\int_{\kappa(n,s)}^sn^{\frac{3}{4}-\frac{2}{p_0}}(1+|x_r^n|^{p_0-3}+|x_{\kappa(n,s)}^n|^{p_0-3})n^{-\frac{1}{2}+\frac{2}{p_0}}|\tilde{\sigma}^n(r,x_{\kappa(n,r)}^n)|^2\,dr\right)^{\frac{p_0}{p_0-1}}\,ds\\
						& +C\sum_{j=1}^mn^{-\frac{p_0}{4}+1}\mathbb{E}\int_0^t\int_{\kappa(n,s)}^s\left|\sum_{j_1=1}^d\int_{\kappa(n,r)}^r L^{n,j}L^{j_1}\sigma(x_{\kappa(n,\gamma)}^n)\,dw_{\gamma}^{j_1}\right|^{p_0}\, dr\, ds\\
						&+C\int_0^t\mathbb{E}|\sigma^n_M(s,x_{\kappa(n,s)}^n)|^{p_0}\, ds,
\end{align*}
for any \(t \in [0,T]\). By Young's inequality,  H\"{o}lder's inequality and Lemma \ref{lemma2},
\begin{align*}
G_{71}(t) \leq \,	&Cn^{\frac{3p_0-4}{4(p_0-2)}-\frac{1}{p_0-1}}\int_0^t\mathbb{E}\int_{\kappa(n,s)}^s(1+|x_r^n|^{p_0}+|x_{\kappa(n,s)}^n|^{p_0})\,dr\,ds\\
						&+Cn^{-\frac{p_0}{2}+1-\frac{1}{p_0-1}}\int_0^t\int_{\kappa(n,s)}^s\mathbb{E}|\tilde{b}^n(r,x_{\kappa(n,r)}^n)|^{p_0}\,dr\,ds\\
						&+Cn\mathbb{E}\int_0^t\int_{\kappa(n,s)}^s(1+|x_r^n|^{p_0}+|x_{\kappa(n,s)}^n|^{p_0})\,dr\,ds\\
						&+Cn^{-\frac{p_0}{4}+1}\int_0^t\int_{\kappa(n,s)}^s\mathbb{E}|\tilde{\sigma}^n(r,x_{\kappa(n,r)}^n)|^{p_0}\, dr\, ds\\
						&+Cn^{\frac{3p_0-8}{4(p_0-2)}}\mathbb{E}\int_0^t\int_{\kappa(n,s)}^s(1+|x_r^n|^{p_0}+|x_{\kappa(n,s)}^n|^{p_0})\,dr\,ds\\
						&+Cn^{-\frac{p_0}{4}+1}\int_0^t\int_{\kappa(n,s)}^s\mathbb{E}|\tilde{\sigma}^n(r,x_{\kappa(n,r)}^n)|^{p_0}\,dr\,ds\\
						&+Cn^{\frac{3p_0-8}{4(p_0-3)}-\frac{1}{p_0-1}}\int_0^t\mathbb{E}\int_{\kappa(n,s)}^s(1+|x_r^n|^{p_0}+|x_{\kappa(n,s)}^n|^{p_0})\,dr\,ds\\
						&+Cn^{-\frac{p_0}{4}+1-\frac{1}{p_0-1}}\int_0^t\int_{\kappa(n,s)}^s\mathbb{E}|\tilde{\sigma}^n(r,x_{\kappa(n,r)}^n)|^{p_0}\,dr\,ds\\
						& +Cn^{-\frac{3p_0}{4}+2}\int_0^t\int_{\kappa(n,s)}^s\int_{\kappa(n,r)}^r \mathbb{E}|L^{n,j}L^{j_1}\sigma(x_{\kappa(n,\gamma)}^n)|^{p_0}\,d{\gamma}\, dr\, ds\\
						&+C+ C\int_0^t\sup_{0 \leq r \leq s}\mathbb{E}|x_r^n|^{p_0}ds,
\end{align*}
for any \(t \in [0,T]\). Due to Corollary \ref{corollary1} and Remark \ref{remark2}, it can be shown that
\begin{equation}\label{g71}
G_{71}(t)	 \leq C+ C\int_0^t\sup_{0 \leq r \leq s}\mathbb{E}|x_r^n|^{p_0}ds,
\end{equation}
for any \(t \in [0,T]\). In order to estimate $G_{72}(t)$, one writes
\begin{align*}
G_{72}(t) 	&=p_0(p_0-1)\mathbb{E}\int_0^t|x_{\kappa(n,s)}^n|^{p_0-2}\sum_{k=1}^d\sum_{v=1}^m\sigma^{n,(k,v)}(x_{\kappa(n,s)}^n)\sum_{j=1}^m\int_{\kappa(n,s)}^s L^{n,j}\sigma^{(k,v)}(x_{\kappa(n,r)}^n)\, dw_r^j\, ds\\
						&\hspace{1em}+p_0(p_0-1)\mathbb{E}\int_0^t|x_{\kappa(n,s)}^n|^{p_0-2}\sum_{k=1}^d\sum_{v=1}^m\sigma^{n,(k,v)}(x_{\kappa(n,s)}^n)\int_{\kappa(n,s)}^s L^{n,0}\sigma^{(k,v)}(x_{\kappa(n,r)}^n)\, dr\, ds\\
						&\hspace{1em}+p_0(p_0-1)\mathbb{E}\int_0^t|x_{\kappa(n,s)}^n|^{p_0-2}\sum_{k=1}^d\sum_{v=1}^m\sigma^{n,(k,v)}(x_{\kappa(n,s)}^n)\\
						& \hspace{5em} \times \sum_{j=1}^m\sum_{j_1=1}^m\int_{\kappa(n,s)}^s \int_{\kappa(n,r)}^r L^{n,j}L^{j_1}\sigma^{(k,v)}(x_{\kappa(n,\gamma)}^n)\,dw_{\gamma}^{j_1}\, dw_r^j\, ds,
\end{align*}
which implies, due to Remark \ref{remark2} and the fact that the first and third terms are zero,
\begin{align*}
G_{72}(t)  \leq \,	&C\mathbb{E}\int_0^t\int_{\kappa(n,s)}^sn(1+|x_{\kappa(n,s)}^n|)^{p_0}\, dr\, ds,
\end{align*}
for any \(t \in [0,T]\). Then, one obtains
\begin{equation}\label{g72}
G_{72}(t)	 \leq C+ C\int_0^t\sup_{0 \leq r \leq s}\mathbb{E}|x_r^n|^{p_0}ds,
\end{equation}
for any \(t \in [0,T]\). Furthermore, substituting \eqref{g71} and \eqref{g72} into \eqref{g7} yields
\[
G_{7}	 \leq C+ C\int_0^t\sup_{0 \leq r \leq s}\mathbb{E}|x_r^n|^{p_0}ds,
\]
for any \(t \in [0,T]\). Therefore, for any \(n \in \mathbb{N}\) and $t \in [0,T]$,
\[
\sup_{0 \leq s \leq t}\mathbb{E}|x_s^n|^{p_0} \leq C+ C\int_0^t\sup_{0 \leq r \leq s}\mathbb{E}|x_r^n|^{p_0}ds<\infty,
\]
%Therefore,
%\[
%\mathbb{E}|x_t^n|^{p_0} \leq C+ C\int_0^t\sup_{0 \leq r \leq s}\mathbb{E}|x_r^n|^{p_0}ds<\infty,
%\]
and applying Gronwall's lemma completes the proof.
\end{proof}

\section{Proof of main result}
%Throughout this section, \(\rho =4\).
\begin{lemma}\label{mvt}
Let \(f: \mathbb{R}^d \rightarrow \mathbb{R}\) be a twice continuously differentiable function. If there exist constants \(\alpha \in \mathbb{R}\), $K>0$ and $\beta \in (0,1]$, such that for any \(x, \bar{x} \in \mathbb{R}^d\), 
\[
|\nabla^2f(x)-\nabla^2f(\bar{x})| \leq K(1+|x|+|\bar{x}|)^{\alpha}|x-\bar{x}|^{\beta},
\]
then, there is a constant \(C>0\) such that for any \(x, \bar{x} \in \mathbb{R}^d\), and \(i = 1, \dots,d\),
\[
\left|\frac{\partial f(x)}{\partial y^{(i)}} - \frac{\partial f(\bar{x})}{\partial y^{(i)}} -\sum_{j = 1}^d \frac{\partial^2 f(\bar{x})}{\partial y^{(i)}\partial y^{(j)}}(x^{(j)} - \bar{x}^{(j)}) \right|\leq C(1+|x|+|\bar{x}|)^{\alpha}|x-\bar{x}|^{1+\beta}.
\]
\end{lemma}
\begin{proof}
One uses the mean value theorem to obtain that, for all  $x, \bar{x} \in \mathbb{R}^d$, \(i = 1, \dots,d\), there exists $q \in [0,1]$, such that
\[
\frac{\partial f(x)}{\partial y^{(i)}} - \frac{\partial f(\bar{x})}{\partial y^{(i)}}  = \sum_{j = 1}^d \frac{\partial^2 f((qx+(1-q)\bar{x})}{\partial y^{(i)}\partial y^{(j)}}(x^{(j)} - \bar{x}^{(j)}).
\]
Then for a fixed \(q \in (0,1)\),
\begin{align*}
&\left|\frac{\partial f(x)}{\partial y^{(i)}} - \frac{\partial f(\bar{x})}{\partial y^{(i)}} -\sum_{j = 1}^d \frac{\partial^2 f(\bar{x})}{\partial y^{(i)}\partial y^{(j)}}(x^{(j)} - \bar{x}^{(j)}) \right|\\
& \hspace{1em} = \left|\sum_{j = 1}^d \frac{\partial^2 f((qx+(1-q)\bar{x})}{\partial y^{(i)}\partial y^{(j)}}(x^{(j)} - \bar{x}^{(j)}) -\sum_{j = 1}^d \frac{\partial^2 f(\bar{x})}{\partial y^{(i)}\partial y^{(j)}}(x^{(j)} - \bar{x}^{(j)}) \right|\\
& \hspace{1em} \leq \sum_{j = 1}^d\left| \frac{\partial^2 f((qx+(1-q)\bar{x})}{\partial y^{(i)}\partial y^{(j)}}- \frac{\partial^2 f(\bar{x})}{\partial y^{(i)}\partial y^{(j)}} \right||x^{(j)} - \bar{x}^{(j)}|\\
& \hspace{1em} \leq C(1+|x|+|\bar{x}|)^{\alpha}|x-\bar{x}|^{1+\beta}.
\end{align*}
\end{proof}
\begin{lemma}\label{lemma6}
Assume \ref{a1} to \ref{a5} hold, then, there exists a constant $C>0$, such that for any \(p\leq \frac{p_0}{2\rho+1}\) and \(n \in \mathbb{N}\),
\[
\sup_{0 \leq t \leq T}\mathbb{E}|b_1^n(t, x_{\kappa(n,t)}^n)|^p \leq Cn^{-p}, \quad 
\sup_{0 \leq t \leq T}\mathbb{E}|b_2^n(t, x_{\kappa(n,t)}^n)|^p \leq Cn^{-\frac{p}{2}},
\]
\[
\sup_{0 \leq t \leq T}\mathbb{E}|\sigma_1^n(t, x_{\kappa(n,t)}^n)|^p \leq Cn^{-\frac{p}{2}}, \quad \sup_{0 \leq t \leq T}\mathbb{E}|\sigma_2^n(t, x_{\kappa(n,t)}^n)|^p \leq Cn^{-p},
\]
\[
\sup_{0 \leq t \leq T}\mathbb{E}|\sigma_3^n(t, x_{\kappa(n,t)}^n)|^p \leq Cn^{-p}.
\]
\end{lemma}
\begin{proof}
By applying  H\"{o}lder's inequality and Remark \ref{remark1}, one obtains, for any \(p\leq \frac{p_0}{2\rho+1}\),
\begin{align*}
\mathbb{E}|b_1^n(t, x_{\kappa(n,t)}^n)|^p 	& = \mathbb{E}\left|\int_{\kappa(n,t)}^t L^{n,0}b(x_{\kappa(n,s)}^n)\,ds\right|^p\\
								& \leq Cn^{-p+1}\int_{\kappa(n,t)}^t\mathbb{E}|L^{n,0}b(x_{\kappa(n,s)}^n)|^{p}\,ds\\
								& \leq Cn^{-p+1}\int_{\kappa(n,t)}^t\mathbb{E}(1+|x_{\kappa(n,s)}^n|)^{(2\rho+1)p}\,ds\\
								& \leq Cn^{-p},
%\mathbb{E}|\sigma_3^n(t, x_{\kappa(n,t)}^n)|^p 	& \leq C\sum_{j=1}^d\sum_{j_1=1}^d\mathbb{E}\left|\int_{\kappa(n,t)}^t \int_{\kappa(n,t)}^sL^{n,j}L^{j_1}\sigma(x_{\kappa(n,s)}^n)\,dw_r^j\,dw_s^{j_1}\right|^p\\
%								& \leq Cn^{-\frac{p}{2}+1}\mathbb{E}\int_{\kappa(n,t)}^t \left|\int_{\kappa(n,t)}^sL^{n,j}L^{j_1}\sigma(x_{\kappa(n,s)}^n)\,dw_r^j\right|^p\,ds\\
%								& \leq Cn^{-p+2}\int_{\kappa(n,t)}^t \mathbb{E}\int_{\kappa(n,t)}^s|L^{n,j}L^{j_1}\sigma(x_{\kappa(n,s)}^n)|^p\,dr\,ds\\
%								& \leq Cn^{-p+2}\int_{\kappa(n,t)}^t \int_{\kappa(n,t)}^s\mathbb{E}(1+|x_{\kappa(n,s)}^n|)^{(\frac{3}{2}\rho +1)p}\,dr\,ds \\
%								& \leq Cn^{-p},
\end{align*}
where the last inequality holds due to Lemma \ref{lemmalp0}. Other results can be proved by using similar arguments.
\end{proof}
\begin{corollary} \label{corollary2}
Assume \ref{a1} to \ref{a5} hold, then, there exists a constant $C>0$, such that for any \(p\leq \frac{p_0}{2\rho+1}\) and \(n \in \mathbb{N}\),
\[
\sup_{0 \leq t \leq T}\mathbb{E}|\tilde{b}^n(t, x_{\kappa(n,t)}^n)|^p \leq C, \quad 
\sup_{0 \leq t \leq T}\mathbb{E}|\tilde{\sigma}^n(t, x_{\kappa(n,t)}^n)|^p \leq C.
\]
\end{corollary}
\begin{lemma}\label{lemma7}
Assume \ref{a1} to \ref{a5} hold, then, there exists a constant $C>0$, such that for any \(p\leq \frac{p_0}{2\rho+1}\) and \(n \in \mathbb{N}\),
\[
\sup_{0 \leq t \leq T}\mathbb{E}|x_t^n -x_{\kappa(n,t)}^n |^p \leq Cn^{-\frac{p}{2}}.
\]
\end{lemma}
\begin{proof} For $p \geq 1$, by using H\"{o}lder's inequality, one obtains
\begin{align*}
\mathbb{E}|x_t^n -x_{\kappa(n,t)}^n |^p 	& \leq C\mathbb{E}\left|\int_{\kappa(n,t)}^t\tilde{b}^n(s, x_{\kappa(n,s)}^n)\,ds\right|^p +C\mathbb{E}\left|\int_{\kappa(n,t)}^t\tilde{\sigma}^n(s, x_{\kappa(n,s)}^n)\,dw_s\right|^p\\
													& \leq n^{-p+1}C\mathbb{E}\int_{\kappa(n,t)}^t |\tilde{b}^n(s, x_{\kappa(n,s)}^n)|^p\,ds+Cn^{-\frac{p}{2}+1}\mathbb{E}\int_{\kappa(n,t)}^t|\tilde{\sigma}^n(s, x_{\kappa(n,s)}^n)|^p\,ds,%\\
												%	& \leq  Cn^{-\frac{p}{2}},
\end{align*}
which by using corollary \ref{corollary2} yields the desired result. As for $p \in (0,1)$, one uses Jensen's inequality to obtain the same result.
\end{proof}
\begin{lemma}\label{lemma8}
Assume \ref{a1} to \ref{a5} hold, then, there exists a constant $C>0$, such that  for any \(p\leq \frac{p_0}{2\rho+1}\) and \(n \in \mathbb{N}\),
\[
\sup_{0 \leq t \leq T}\mathbb{E}|b(x_{\kappa(n,t)}^n) -b^n(x_{\kappa(n,t)}^n) |^p \leq Cn^{-\frac{3}{2}p}, \quad \sup_{0 \leq t \leq T}\mathbb{E}|\sigma(x_{\kappa(n,t)}^n) -\sigma^n(x_{\kappa(n,t)}^n) |^p \leq Cn^{-\frac{3}{2}p}
\]
\end{lemma}
\begin{proof} We have the following expression,
\[
|b(x_{\kappa(n,t)}^n) -b^n(x_{\kappa(n,t)}^n)|= n^{-\frac{3}{2}}\frac{|x_{\kappa(n,t)}^n|^{3\rho}|b(x_{\kappa(n,t)}^n)|}{1+n^{-\frac{3}{2}}|x_{\kappa(n,t)}^n|^{3\rho}} \leq n^{-\frac{3}{2}}(1+|x_{\kappa(n,t)}^n|)^{4\rho+1},
\]
and then by using Lemma \ref{lemmalp0} and the same argument for $\sigma$ completes the proof.
\end{proof}
\begin{lemma}\label{lemma9}
Assume \ref{a1} to \ref{a5} hold and \(p_0 \geq  2(5\rho +1)\). Then, there exists a constant $C>0$, such that for any \(n \in \mathbb{N}\),
\[
\sup_{0\leq t \leq T}\mathbb{E}|\sigma(x_t^n)-\sigma(x_{\kappa(n,t)}^n)-\sigma_M^n(t,x_{\kappa(n,t)}^n)|^2 \leq Cn^{-(2+\beta)}.
\]
\end{lemma}
\begin{proof}
For every \(k = 1, \dots, d,\) \(v = 1, \dots, m\), applying It\^o's formula to \(\sigma^{(k,v)}(x_t^n)-\sigma^{(k,v)}(x_{\kappa(n,t)}^n)\) gives, almost surely,
\begin{align*}
&\sigma^{(k,v)}(x_t^n)-\sigma^{(k,v)}(x_{\kappa(n,t)}^n)\\
&=\,	  \sum_{i = 1}^d \int_{\kappa(n,t)}^t \frac{\partial \sigma^{(k,v)}(x_s^n)}{\partial x^{(i)}}\tilde{b}^{n,(i)}(s, x_{\kappa(n,s)}^n)\,ds  + \sum_{i = 1}^d \sum_{j = 1}^m \int_{\kappa(n,t)}^t \frac{\partial \sigma^{(k,v)}(x_s^n)}{\partial x^{(i)}}\tilde{\sigma}^{n,(i,j)}(s, x_{\kappa(n,s)}^n)\,dw_s^j\\
																		&\hspace{1em}+ \frac{1}{2}\sum_{i,l = 1}^d \sum_{j = 1}^m\int_{\kappa(n,t)}^t \frac{\partial^2 \sigma^{(k,v)}(x_s^n)}{\partial x^{(i)} \partial x^{(l)}}\tilde{\sigma}^{n,(i,j)}(s, x_{\kappa(n,s)}^n)\tilde{\sigma}^{n,(l,j)}(s, x_{\kappa(n,s)}^n)\,ds =\,	 \sum_{i=1}^{12}	J_i(t)										
\end{align*}	
where													
\begin{align*}																		
J_1(t)& =  \sum_{i = 1}^d \int_{\kappa(n,t)}^t\left( \frac{\partial \sigma^{(k,v)}(x_s^n)}{\partial x^{(i)}}-\frac{\partial \sigma^{(k,v)}(x_{\kappa(n,s)}^n)}{\partial x^{(i)}}\right)b^{n,(i)}( x_{\kappa(n,s)}^n)\,ds,\\
J_2(t)&= \sum_{i = 1}^d \int_{\kappa(n,t)}^t \frac{\partial \sigma^{(k,v)}(x_{\kappa(n,s)}^n)}{\partial x^{(i)}}b^{n,(i)}( x_{\kappa(n,s)}^n)\,ds,\\
J_3(t)&=\sum_{i = 1}^d \int_{\kappa(n,t)}^t \frac{\partial \sigma^{(k,v)}(x_s^n)}{\partial x^{(i)}}(b_1^{n,(i)}(s, x_{\kappa(n,s)}^n)+b_2^{n,(i)}(s, x_{\kappa(n,s)}^n))\,ds,\\
J_4(t)&= \sum_{i = 1}^d \sum_{j = 1}^m \int_{\kappa(n,t)}^t \left( \frac{\partial \sigma^{(k,v)}(x_s^n)}{\partial x^{(i)}}-\frac{\partial \sigma^{(k,v)}(x_{\kappa(n,s)}^n)}{\partial x^{(i)}} \right. \\
																			&\hspace{11em} \left. -\sum_{l=1}^d\frac{\partial^2 \sigma^{(k,v)}(x_{\kappa(n,s)}^n)}{\partial x^{(i)} \partial x^{(l)}}(x_s^{n,(l)}-x_{\kappa(n,s)}^{n,(l)})\right) \sigma^{n,(i,j)}(x_{\kappa(n,s)}^n)\,dw_s^j,\\
J_5(t)&= \sum_{i = 1}^d \sum_{j = 1}^m \int_{\kappa(n,t)}^t \sum_{l=1}^d\frac{\partial^2 \sigma^{(k,v)}(x_{\kappa(n,s)}^n)}{\partial x^{(i)} \partial x^{(l)}}\left(\int_{\kappa(n,s)}^s \tilde{b}^{n,(l)}(r, x_{\kappa(n,r)}^n)\,dr\right. \\
																			&  \hspace{11em} +\left. \sum_{j_1=1}^m\int_{\kappa(n,s)}^s \sigma_M^{n,(l,j_1)}(r, x_{\kappa(n,r)}^n)\,dw_r^{j_1}\right)\sigma^{n,(i,j)}(x_{\kappa(n,s)}^n)\,dw_s^j,\\
J_6(t)&= \sum_{i = 1}^d \sum_{j = 1}^m \int_{\kappa(n,t)}^t \sum_{l=1}^d\frac{\partial^2 \sigma^{(k,v)}(x_{\kappa(n,s)}^n)}{\partial x^{(i)} \partial x^{(l)}} \sum_{j_1=1}^m\int_{\kappa(n,s)}^s \sigma^{n,(l,j_1)}(x_{\kappa(n,r)}^n)\,dw_r^{j_1}\sigma^{n,(i,j)}(x_{\kappa(n,s)}^n)\,dw_s^j,\\
J_7(t)&= \sum_{i = 1}^d \sum_{j = 1}^m \int_{\kappa(n,t)}^t\left( \frac{\partial \sigma^{(k,v)}(x_s^n)}{\partial x^{(i)}}-\frac{\partial \sigma^{(k,v)}(x_{\kappa(n,s)}^n)}{\partial x^{(i)}}\right)\sigma_1^{n,(i,j)}(s, x_{\kappa(n,s)}^n)\,dw_s^{j}	,\\
J_8(t)&= \sum_{i = 1}^d \sum_{j = 1}^m \int_{\kappa(n,t)}^t\frac{\partial \sigma^{(k,v)}(x_{\kappa(n,s)}^n)}{\partial x^{(i)}}(\sigma^{n,(i,j)}(x_{\kappa(n,s)}^n)+\sigma_1^{n,(i,j)}(s, x_{\kappa(n,s)}^n))\,dw_s^j,\\			
J_9(t)&= \sum_{i = 1}^d \sum_{j = 1}^m \int_{\kappa(n,t)}^t\frac{\partial \sigma^{(k,v)}(x_s^n)}{\partial x^{(i)}}(\sigma_2^{n,(i,j)}(s,x_{\kappa(n,s)}^n)+\sigma_3^{n,(i,j)}(s, x_{\kappa(n,s)}^n))\,dw_s^j,\\
J_{10}(t)&= \frac{1}{2}\sum_{i,l = 1}^d \sum_{j = 1}^m\int_{\kappa(n,t)}^t \left(\frac{\partial^2 \sigma^{(k,v)}(x_s^n)}{\partial x^{(i)} \partial x^{(l)}}-\frac{\partial^2 \sigma^{(k,v)}(x_{\kappa(n,s)}^n)}{\partial x^{(i)} \partial x^{(l)}}\right)\sigma^{n,(i,j)}( x_{\kappa(n,s)}^n)\sigma^{n,(l,j)}( x_{\kappa(n,s)}^n)\,ds,\\	
J_{11}(t)&= \frac{1}{2}\sum_{i,l = 1}^d \sum_{j = 1}^m\int_{\kappa(n,t)}^t \frac{\partial^2 \sigma^{(k,v)}(x_{\kappa(n,s)}^n)}{\partial x^{(i)} \partial x^{(l)}}\sigma^{n,(i,j)}( x_{\kappa(n,s)}^n)\sigma^{n,(l,j)}( x_{\kappa(n,s)}^n)\,ds,\\		
J_{12}(t)&= \frac{1}{2}\sum_{i,l = 1}^d \sum_{j = 1}^m\int_{\kappa(n,t)}^t \frac{\partial^2 \sigma^{(k,v)}(x_s^n)}{\partial x^{(i)} \partial x^{(l)}}(\sigma^{n,(i,j)}( x_{\kappa(n,s)}^n)\sigma_M^{n,(l,j)}(s, x_{\kappa(n,s)}^n)\\
																			&\hspace{15em} +\sigma_M^{n,(i,j)}(s, x_{\kappa(n,s)}^n)\tilde{\sigma}^{n,(l,j)}(s, x_{\kappa(n,s)}^n))\,ds.
\end{align*}
It can be observed that
\begin{align*}
&\mathbb{E}|J_2(t)+J_6(t)+J_8(t)+J_{11}(t) -\sigma_M^{n,(k,v)}(t, x_{\kappa(n,t)}^n)|^2\\
&\leq 2\mathbb{E}|J_2(t)+J_{11}(t) -\sigma_2^{n,(k,v)}(t, x_{\kappa(n,t)}^n)|^2+2\mathbb{E}|J_6(t)+J_8(t) -\sigma_1^{n,(k,v)}(t, x_{\kappa(n,t)}^n)-\sigma_3^{n,(k,v)}(t, x_{\kappa(n,t)}^n)|^2\\
&\leq C\sum_{i,l = 1}^d \sum_{j= 1}^m\mathbb{E}\left|-\frac{n^{-3/2}|x_{\kappa(n,t)}^n|^{3\rho}}{(1+n^{-3/2}|x_{\kappa(n,t)}^n|^{3\rho})^2}\int_{\kappa(n,t)}^t \frac{\partial^2 \sigma^{(k,v)}(x_{\kappa(n,s)}^n)}{\partial x^{(i)} \partial x^{(l)}}\sigma^{(i,j)}( x_{\kappa(n,s)}^n)\sigma^{(l,j)}( x_{\kappa(n,s)}^n)\,ds\right|^2\\%\sum_{i,l = 1}^d \mathbb{E}\left|-\frac{n^{-3/2}|x_{\kappa(n,s)}^n|^{3\rho}}{(1+n^{-3/2}|x_{\kappa(n,s)}^n|^{3\rho})^2}\int_{\kappa(n,t)}^t \left|\frac{\partial^2 \sigma^{(k,v)}(x_{\kappa(n,s)}^n)}{\partial x^{(i)} \partial x^{(l)}}\right||\sigma( x_{\kappa(n,s)}^n)|^2\,ds\right|^2\\%
&\hspace{1em}+2\mathbb{E}\left|-\frac{n^{-3/2}|x_{\kappa(n,t)}^n|^{3\rho}}{(1+n^{-3/2}|x_{\kappa(n,t)}^n|^{3\rho})^2}\sum_{i,l = 1}^d \sum_{j, j_1= 1}^m \int_{\kappa(n,t)}^t \frac{\partial^2 \sigma^{(k,v)}(x_{\kappa(n,s)}^n)}{\partial x^{(i)} \partial x^{(l)}} \right.\\
&\hspace{10em}\left. \times \int_{\kappa(n,s)}^s \sigma^{(l,j_1)}(x_{\kappa(n,r)}^n)\,dw_r^{j_1}\sigma^{(i,j)}(x_{\kappa(n,s)}^n)\,dw_s^j\right|^2,
\end{align*}
which implies due to Remark \ref{remark1} and Lemma \ref{lemmalp0} that
\begin{align*}
&\mathbb{E}|J_2(t)+J_6(t)+J_8(t)+J_{11}(t) -\sigma_M^{n,(k,v)}(t, x_{\kappa(n,t)}^n)|^2\\
&  \leq 	Cn^{-3}\mathbb{E}|n^{-1}|x_{\kappa(n,t)}^n|^{3\rho}(1+|x_{\kappa(n,t)}^n|^{3/2\rho+1})|^2 +Cn^{-5}\mathbb{E}||x_{\kappa(n,t)}^n|^{3\rho}(1+|x_{\kappa(n,t)}^n|^{3/2\rho+1})|^2\leq Cn^{-5},
\end{align*}
for \(p_0 \geq 9\rho +2\). % In addition, note that \(J_6(t)+J_8(t) =\sigma_1^{n,(i,j)}(t, x_{\kappa(n,t)}^n)+\sigma_3^{n,(i,j)}(t, x_{\kappa(n,t)}^n) \).
Then, one obtains the following
\begin{align*}
&\mathbb{E}|\sigma^{(k,v)}(x_t^n)-\sigma^{(k,v)}(x_{\kappa(n,t)}^n)-\sigma_M^{n,(k,v)}(t,x_{\kappa(n,t)}^n)|^2 \\
& \hspace{1em} \leq 2\mathbb{E}|J_1(t)+J_3(t)+J_4(t)+J_5(t)+J_7(t)+J_9(t)+J_{10}(t)+J_{12}(t)|^2\\
& \hspace{2em} +2\mathbb{E}|J_2(t)+J_6(t)+J_8(t)+J_{11}(t) -\sigma_M^{n,(k,v)}(t, x_{\kappa(n,t)}^n)|^2\\
%& \leq \mathbb{E}|J_1(t)|^2+\mathbb{E}|J_3(t)|^2+\mathbb{E}|J_4(t)|^2+\mathbb{E}|J_5(t)|^2+\mathbb{E}|J_7(t)|^2+\mathbb{E}|J_9(t)|^2+\mathbb{E}|J_{10}(t)|^2+\mathbb{E}|J_{12}(t)|^2 +Cn^{-5}.
&  \hspace{1em} \leq C(\mathbb{E}|J_1(t)|^2+\mathbb{E}|J_3(t)|^2+\mathbb{E}|J_4(t)|^2+\mathbb{E}|J_5(t)|^2+\mathbb{E}|J_7(t)|^2\\
& \hspace{2em}+\mathbb{E}|J_9(t)|^2+\mathbb{E}|J_{10}(t)|^2+\mathbb{E}|J_{12}(t)|^2)+Cn^{-5},
\end{align*}
for any \(t \in [0,T]\). By using Cauchy-Schwarz inequality, \(\mathbb{E}|J_1(t)|^2\) can be estimated as
\begin{align*}
\mathbb{E}|J_1(t)|^2		& \leq Cn^{-1}\sum_{i = 1}^d \int_{\kappa(n,t)}^t\mathbb{E}\left|\left( \frac{\partial \sigma^{(k,v)}(x_s^n)}{\partial x^{(i)}}-\frac{\partial \sigma^{(k,v)}(x_{\kappa(n,s)}^n)}{\partial x^{(i)}}\right)b^{n}( x_{\kappa(n,s)}^n)\right|^2\,ds,
\end{align*}
which by using Young's inequality, Remark \ref{remark1} and H\"{o}lder's inequality yields
\begin{align*}
\mathbb{E}|J_1(t)|^2		&\leq  Cn^{-1}\int_{\kappa(n,t)}^t\mathbb{E}(1+|x_s^n|+|x_{\kappa(n,s)}^n|)^{\rho-2}(1+|x_{\kappa(n,s)}^n|)^{2\rho+2}|x_s^n-x_{\kappa(n,s)}^n|^2\,ds\\
					&\leq  Cn^{-1}\int_{\kappa(n,t)}^t\mathbb{E}(1+|x_s^n|^{3\rho}+|x_{\kappa(n,s)}^n|^{3\rho})|x_s^n-x_{\kappa(n,s)}^n|^2\,ds\\
					& \leq  Cn^{-1}\int_{\kappa(n,t)}^t\left(\mathbb{E}(1+|x_s^n|^{p_0}+|x_{\kappa(n,s)}^n|^{p_0})\right)^{\frac{3\rho}{p_0}}\left(\mathbb{E}|x_s^n-x_{\kappa(n,s)}^n|^{\frac{2p_0}{p_0-3\rho}}\right)^{\frac{p_0-3\rho}{p_0}}\,ds,
\end{align*}
for any \(t \in [0,T]\). One uses Lemma \ref{lemmalp0} and Lemma \ref{lemma7} to obtain
\[
\mathbb{E}|J_1(t)|^2	\leq Cn^{-3},
\]
 %%%%%%%%%%%%%%%%%%%%%%%%%%%%%%%%%%%%%%%%%%%%%%%%%%%%
for every \(n \in \mathbb{N}\). To estimate \(\mathbb{E}|J_3(t)|^2\), one applies Cauchy-Schwarz inequality and Remark \ref{remark1} to obtain
\begin{align*}
\mathbb{E}|J_3(t)|^2		& \leq Cn^{-1}\int_{\kappa(n,t)}^t\mathbb{E} (1+|x_s^n|)^{\rho}(|b_1^{n}(s, x_{\kappa(n,s)}^n)|^2+|b_2^{n}(s, x_{\kappa(n,s)}^n)|^2)\,ds,
\end{align*}
which implies due to H\"{o}lder's inequality
\begin{align*}
\mathbb{E}|J_3(t)|^2		& \leq Cn^{-1}\int_{\kappa(n,t)}^t\left(\mathbb{E}(1+|x_s^n|^{p_0})\right)^{\frac{\rho}{p_0}}(\mathbb{E}(|b_1^{n}(s, x_{\kappa(n,s)}^n)|^{\frac{2p_0}{p_0-\rho}}+|b_2^{n}(s, x_{\kappa(n,s)}^n)|^{\frac{2p_0}{p_0-\rho}}))^{\frac{p_0-\rho}{p_0}}\,ds,
\end{align*}
 for any \(t \in [0,T]\). By Lemma \ref{lemmalp0} and Lemma \ref{lemma6}, it becomes
 \[
 \mathbb{E}|J_3(t)|^2		\leq Cn^{-3},
 \]
  %%%%%%%%%%%%%%%%%%%%%%%%%%%%%%%%%%%%%%%%%%%%%%%%%%%%
  for every \(n \in \mathbb{N}\). As for \(\mathbb{E}|J_4(t)|^2\), by using Young's inequality, Cauchy-Schwarz inequality, Remark \ref{remark1} and Lemma \ref{mvt}, one obtains
\begin{align*}
\mathbb{E}|J_4(t)|^2	&\leq C\int_{\kappa(n,t)}^t \mathbb{E}((1+|x_s^n|+|x_{\kappa(n,s)}^n|)^{\rho-4}(1+|x_{\kappa(n,s)}^n|)^{\rho+2}|x_s^n-x_{\kappa(n,s)}^n|^{2+2\beta}\,ds\\
				&\leq C\int_{\kappa(n,t)}^t \mathbb{E}((1+|x_s^n|+|x_{\kappa(n,s)}^n|)^{2\rho-2}|x_s^n-x_{\kappa(n,s)}^n|^{2+2\beta})\,ds,
\end{align*}
which implies due to H\"{o}lder's inequality
\begin{align*}
\mathbb{E}|J_4(t)|^2	&\leq C\int_{\kappa(n,t)}^t \left(\mathbb{E}(1+|x_s^n|^{p_0}+|x_{\kappa(n,s)}^n|^{p_0})\right)^{\frac{2\rho-2}{p_0}}\left(\mathbb{E}|x_s^n-x_{\kappa(n,s)}^n|^{\frac{(2+2\beta)p_0}{p_0-2\rho+2}}\right)^{\frac{p_0-2\rho+2}{p_0}}\,ds,
\end{align*}
 for any \(t \in [0,T]\).  Then, applying Lemma \ref{lemma7} and Lemma \ref{lemmalp0} yield
\[
 \mathbb{E}|J_4(t)|^2		\leq Cn^{-(2+\beta)},
 \]
  %%%%%%%%%%%%%%%%%%%%%%%%%%%%%%%%%%%%%%%%%%%%%%%%%%%%
 for every \(n \in \mathbb{N}\).
 In order to estimate \(\mathbb{E}|J_5(t)|^2\),  one uses Young's inequality and Cauchy-Schwarz inequality to obtain
 \begin{align*}
\mathbb{E}|J_5(t)|^2 & \leq C\int_{\kappa(n,t)}^t \mathbb{E}\left|\int_{\kappa(n,s)}^s |\tilde{b}^{n}(r, x_{\kappa(n,r)}^n)|\,dr +\left|\sum_{j_1=1}^m\int_{\kappa(n,s)}^s \sigma_M^{n,(i,j_1)}(r, x_{\kappa(n,r)}^n)\,dw_r^{j_1}\right|\right|^2\\
				&\hspace{11em} \times (1+|x_{\kappa(n,s)}^n|)^{2\rho}\,ds,
\end{align*}
 which, by applying H\"{o}lder's inequality, yields 
 \begin{align*}
\mathbb{E}|J_5(t)|^2	& \leq C\int_{\kappa(n,t)}^t \left(n^{-\frac{2p_0}{p_0-2\rho}+1}\int_{\kappa(n,s)}^s \mathbb{E}|\tilde{b}^{n}(r, x_{\kappa(n,r)}^n)|^{\frac{2p_0}{p_0-2\rho}}\,ds\right.\\
				&\hspace{1em}\left.+n^{-\frac{p_0}{p_0-2\rho}+1}\int_{\kappa(n,s)}^s\mathbb{E}|\sigma_M^{n}(r, x_{\kappa(n,r)}^n)|^{\frac{2p_0}{p_0-2\rho}}\,ds\right)^{\frac{p_0-2\rho}{p_0}}\left(\mathbb{E}(1+|x_{\kappa(n,s)}^n|^{p_0})\right)^{\frac{2\rho}{p_0}}\,ds,
\end{align*}
for any \(t \in [0,T]\). One uses Corollary \ref{corollary2} and Lemma \ref{lemma6} to obtain
\[
 \mathbb{E}|J_5(t)|^2		\leq Cn^{-3},
 \]
 %%%%%%%%%%%%%%%%%%%%%%%%%%%%%%%%%%%%%%%%%%%%%%%%%%%%
 for every \(n \in \mathbb{N}\). As for \(\mathbb{E}|J_7(t)|^2\), it can be estimated by using Cauchy-Schwarz inequality as follows
\begin{align*}
\mathbb{E}|J_7(t)|^2	& \leq C\sum_{i = 1}^d \sum_{j = 1}^m \int_{\kappa(n,t)}^t\mathbb{E}\left|\frac{\partial \sigma^{(k,v)}(x_s^n)}{\partial x^{(i)}}-\frac{\partial \sigma^{(k,v)}(x_{\kappa(n,s)}^n)}{\partial x^{(i)}}\right|^2|\sigma_1^{n}(s, x_{\kappa(n,s)}^n)|^2\,ds,
\end{align*}
 which yields by using Remark \ref{remark1} and H\"{o}lder's inequality
\begin{align*}
\mathbb{E}|J_7(t)|^2	& \leq C \int_{\kappa(n,t)}^t\mathbb{E}(1+|x_s^n|+|x_{\kappa(n,s)}^n|)^{\rho-2}|x_s^n - x_{\kappa(n,s)}^n|^2|\sigma_1^{n}(s, x_{\kappa(n,s)}^n)|^2\,ds\\
				& \leq C\int_{\kappa(n,t)}^t\left(\mathbb{E}(1+|x_s^n|+|x_{\kappa(n,s)}^n|)^{p_0}\right)^{\frac{\rho-2}{p_0}}\\
				&\hspace{5em} \times \left(\mathbb{E}|x_s^n-x_{\kappa(n,s)}^n)|^{\frac{2p_0}{p_0-\rho+2}}|\sigma_1^{n,}(s, x_{\kappa(n,s)}^n)|^{\frac{2p_0}{p_0-\rho+2}}\right)^{\frac{p_0-\rho+2}{p_0}}\,ds,
\end{align*}
for \(\rho >2\), and any \(t \in [0,T]\). Then, one can apply Cauchy-Schwarz inequality and Lemma \ref{lemmalp0} to obtain
\begin{align*}
\mathbb{E}|J_7(t)|^2	& \leq C\int_{\kappa(n,t)}^t \left(\mathbb{E}|x_s^n-x_{\kappa(n,s)}^n)|^{\frac{4p_0}{p_0-\rho+2}}\mathbb{E}|\sigma_1^{n,}(s, x_{\kappa(n,s)}^n)|^{\frac{4p_0}{p_0-\rho+2}}\right)^{\frac{p_0-\rho+2}{2p_0}}\,ds,
\end{align*}
Thus, applying  Lemma \ref{lemma7} and Lemma \ref{lemma6} give the following estimate
\[
 \mathbb{E}|J_7(t)|^2		\leq Cn^{-3},
 \]
 %%%%%%%%%%%%%%%%%%%%%%%%%%%%%%%%%%%%%%%%%%%%%%%%%%%%
 for every \(n \in \mathbb{N}\). Note that, for the case that \(\rho=2\), one obtains the same result immediately by using Cauchy-Schwarz inequality. As for \(\mathbb{E}|J_9(t)|^2\), applying Remark \ref{remark1} yields
\begin{align*}
\mathbb{E}|J_9(t)|^2	 & \leq C\int_{\kappa(n,t)}^t\mathbb{E}(1+|x_s^n|)^{\rho}(|\sigma_2^{n}(s,x_{\kappa(n,s)}^n)|^2+|\sigma_3^{n}(s, x_{\kappa(n,s)}^n)|^2)\,ds
			%	& \leq C\int_{\kappa(n,t)}^t\mathbb{E}(1+|x_s^n|)^{\rho}(|\sigma_2^{n}(s,x_{\kappa(n,s)}^n)|^2+|\sigma_3^{n}(s, x_{\kappa(n,s)}^n)|^2)\,ds,
\end{align*}
which by applying H\"{o}lder's inequality gives
\begin{align*}
\mathbb{E}|J_9(t)|^2	& \leq C\int_{\kappa(n,t)}^t\left(\mathbb{E}(1+|x_s^n|^{p_0})\right)^{\frac{\rho}{p_0}}(\mathbb{E}(|\sigma_2^{n}(s,x_{\kappa(n,s)}^n)|^{\frac{2p_0}{p_0-\rho}}+|\sigma_3^{n}(s, x_{\kappa(n,s)}^n)|^{\frac{2p_0}{p_0-\rho}}))^{\frac{p_0-\rho}{p_0}}\,ds,
\end{align*}
for any \(t \in [0,T]\). By Lemma \ref{lemma6}, one obtains
\[
 \mathbb{E}|J_9(t)|^2		\leq Cn^{-3},
 \]
 %%%%%%%%%%%%%%%%%%%%%%%%%%%%%%%%%%%%%%%%%%%%%%%%%%%%
 for every \(n \in \mathbb{N}\). To estimate \(\mathbb{E}|J_{10}(t)|^2\), one uses Young's inequality and Remark \ref{remark1} to obtain
\begin{align*}
\mathbb{E}|J_{10}(t)|^2	& \leq Cn^{-1}\int_{\kappa(n,t)}^t\mathbb{E}(1+|x_s^n|+|x_{\kappa(n,s)}^n|)^{3\rho}|x_s^n - x_{\kappa(n,s)}^n|^{2\beta}\,ds
%					& \leq Cn^{-1}\int_{\kappa(n,t)}^t\mathbb{E}(1+|x_s^n|+|x_{\kappa(n,s)}^n|)^{3\rho}|x_s^n - x_{\kappa(n,s)}^n|^{2\beta}\,ds,
\end{align*}
which implies due to H\"{o}lder's inequality
\begin{align*}
\mathbb{E}|J_{10}(t)|^2	& \leq Cn^{-1}\int_{\kappa(n,t)}^t\left(\mathbb{E}(1+|x_{\kappa(n,s)}^n|)^{p_0}\right)^{\frac{3\rho}{p_0}}\left(\mathbb{E}|x_s^n - x_{\kappa(n,s)}^n|^{\frac{2\beta p_0}{p_0-3\rho}}\right)^{\frac{p_0-3\rho}{p_0}}\,ds,
\end{align*}
for any \(t \in [0,T]\). Lemma \ref{lemma7} is used to obtain
\[
 \mathbb{E}|J_{10}(t)|^2		\leq Cn^{-(2+\beta)},
 \]
 %%%%%%%%%%%%%%%%%%%%%%%%%%%%%%%%%%%%%%%%%%%%%%%%%%%%
 for every \(n \in \mathbb{N}\). Finally for \(\mathbb{E}|J_{12}(t)|^2\), applying Young's inequality, Cauchy-Schwarz inequality and Remark \ref{remark1} yield
\begin{align*}
\mathbb{E}|J_{12}(t)|^2	& \leq Cn^{-1}\int_{\kappa(n,t)}^t\mathbb{E}(1+|x_s^n|)+|x_{\kappa(n,s)}^n|)^{2\rho}|\sigma_M^{n}(s, x_{\kappa(n,s)}^n)|^2\,ds \\
					& \hspace{1em} + Cn^{-1}\int_{\kappa(n,t)}^t\mathbb{E}(1+|x_s^n|)^{\rho-2}|\tilde{\sigma}^{n}( x_{\kappa(n,s)}^n)|^2|\sigma_M^{n}(s, x_{\kappa(n,s)}^n)|^2\,ds,
\end{align*}
which implies due to H\"{o}lder's inequality
\begin{align*}
\mathbb{E}|J_{12}(t)|^2	& \leq Cn^{-1}\int_{\kappa(n,t)}^t\left(\mathbb{E}(1+|x_s^n|^{p_0}+|x_{\kappa(n,s)}^n|^{p_0})\right)^{\frac{2\rho}{p_0}}\left(\mathbb{E}|\sigma_M^{n}(s, x_{\kappa(n,s)}^n)|^{\frac{2p_0}{p_0-2\rho}}\right)^{\frac{p_0-2\rho}{p_0}}\,ds\\
					& \hspace{1em} +Cn^{-1}\int_{\kappa(n,t)}^t\left(\mathbb{E}(1+|x_s^n|)^{p_0}\right)^{\frac{\rho-2}{p_0}}\left(\mathbb{E}|\tilde{\sigma}^{n}( x_{\kappa(n,s)}^n)|^{\frac{2p_0}{p_0-\rho+2}}|\sigma_M^{n}( x_{\kappa(n,s)}^n)|^{\frac{2p_0}{p_0-\rho+2}}\right)^{\frac{p_0-\rho+2}{p_0}}\,ds,
\end{align*}
for any \(t \in [0,T]\). Applying Lemma \ref{lemmalp0} and Lemma \ref{lemma6} to the first term and Cauchy-Schwarz inequality to the second term give
\begin{align*}
\mathbb{E}|J_{12}(t)|^2	& \leq Cn^{-3}+Cn^{-1}\int_{\kappa(n,t)}^t\left(\mathbb{E}|\tilde{\sigma}^{n}( x_{\kappa(n,s)}^n)|^{\frac{4p_0}{p_0-\rho+2}}\mathbb{E}|\sigma_M^{n}( x_{\kappa(n,s)}^n)|^{\frac{4p_0}{p_0-\rho+2}}\right)^{\frac{p_0-\rho+2}{2p_0}}\,ds,
\end{align*}
which by using Lemma \ref{lemma6} yields the desired result, i.e.
\[
 \mathbb{E}|J_{12}(t)|^2		\leq Cn^{-3},
 \]
 %%%%%%%%%%%%%%%%%%%%%%%%%%%%%%%%%%%%%%%%%%%%%%%%%%%%
 for every \(n \in \mathbb{N}\). Therefore, one obtains, for any \(n \in \mathbb{N}\), \(\beta \in (0,1]\) and \(p_0 \geq 10\rho+2\),
 \[
 \sup_{0\leq t \leq T}\mathbb{E}|\sigma(x_t^n)-\sigma(x_{\kappa(n,t)}^n)-\sigma_M^n(t,x_{\kappa(n,t)}^n)|^2 \leq Cn^{-(2+\beta)}+Cn^{-3}+Cn^{-5} \leq Cn^{-(2+\beta)}.
 \]
\end{proof}
\begin{lemma}\label{lemma10}
Assume \ref{a1} to \ref{a5} hold and \(p_0 \geq 2(5\rho +1)\). Then, there exists a constant $C>0$, such that for any \(n \in \mathbb{N}\),
\[
\sup_{0\leq t \leq T}\mathbb{E}|b(x_t^n)-b(x_{\kappa(n,t)}^n)-b_1^n(t,x_{\kappa(n,t)}^n)-b_2^n(t,x_{\kappa(n,t)}^n)|^2 \leq Cn^{-2}.
\]
\end{lemma}
\begin{proof}
For every \(k = 1, \dots, d\), applying It\^o's formula to \(b^{(k)}(x_t^n)-b^{(k)}(x_{\kappa(n,t)}^n)\) gives, almost surely,
\begin{align}\label{itob}
\begin{split}		
&b^{(k)}(x_t^n)-b^{(k)}(x_{\kappa(n,t)}^n)\\
&=\,	  \sum_{i = 1}^d \int_{\kappa(n,t)}^t \frac{\partial b^{(k)}(x_s^n)}{\partial x^{(i)}}\tilde{b}^{n,(i)}(s, x_{\kappa(n,s)}^n)\,ds+ \sum_{i = 1}^d \sum_{j = 1}^m \int_{\kappa(n,t)}^t \frac{\partial b^{(k)}(x_s^n)}{\partial x^{(i)}}\tilde{\sigma}^{n,(i,j)}(s, x_{\kappa(n,s)}^n)\,dw_s^j\\
																		&\hspace{1em}+ \frac{1}{2}\sum_{i,l = 1}^d \sum_{j = 1}^m\int_{\kappa(n,t)}^t \frac{\partial^2 b^{(k)}(x_s^n)}{\partial x^{(i)} \partial x^{(l)}}\tilde{\sigma}^{n,(i,j)}(s, x_{\kappa(n,s)}^n)\tilde{\sigma}^{n,(l,j)}(s, x_{\kappa(n,s)}^n)\,ds\\
																		&=\, \sum_{i=1}^9 I_i(t),	
\end{split}
\end{align}	
where
\begin{align*}													
I_1(t)& =  \sum_{i = 1}^d \int_{\kappa(n,t)}^t\left( \frac{\partial b^{(k)}(x_s^n)}{\partial x^{(i)}}-\frac{\partial b^{(k)}(x_{\kappa(n,s)}^n)}{\partial x^{(i)}}\right)b^{n,(i)}( x_{\kappa(n,s)}^n)\,ds,\\
I_2(t)& = \sum_{i = 1}^d \int_{\kappa(n,t)}^t \frac{\partial b^{(k)}(x_{\kappa(n,s)}^n)}{\partial x^{(i)}}b^{n,(i)}( x_{\kappa(n,s)}^n)\,ds,\\
I_3(t)& =  \sum_{i = 1}^d \int_{\kappa(n,t)}^t \frac{\partial b^{(k)}(x_s^n)}{\partial x^{(i)}}(b_1^{n,(i)}(s, x_{\kappa(n,s)}^n)+b_2^{n,(i)}(s, x_{\kappa(n,s)}^n))\,ds,\\
I_4(t)& =  \sum_{i = 1}^d \sum_{j = 1}^m \int_{\kappa(n,t)}^t \left( \frac{\partial b^{(k)}(x_s^n)}{\partial x^{(i)}}-\frac{\partial b^{(k)}(x_{\kappa(n,s)}^n)}{\partial x^{(i)}} \right) \sigma^{n,(i,j)}(x_{\kappa(n,s)}^n)\,dw_s^j,\\
I_5(t)& = \sum_{i = 1}^d \sum_{j = 1}^m \int_{\kappa(n,t)}^t\frac{\partial b^{(k)}(x_{\kappa(n,s)}^n)}{\partial x^{(i)}}\sigma^{n,(i,j)}(x_{\kappa(n,s)}^n)\,dw_s^j,\\			
I_6(t)& = \sum_{i = 1}^d \sum_{j = 1}^m \int_{\kappa(n,t)}^t\frac{\partial b^{(k)}(x_s^n)}{\partial x^{(i)}}\sigma_M^{n,(i,j)}(s,x_{\kappa(n,s)}^n)\,dw_s^j,\\
I_7(t)& = \frac{1}{2}\sum_{i,l = 1}^d \sum_{j = 1}^m\int_{\kappa(n,t)}^t \left(\frac{\partial^2 b^{(k)}(x_s^n)}{\partial x^{(i)} \partial x^{(l)}}-\frac{\partial^2 b^{(k)}(x_{\kappa(n,s)}^n)}{\partial x^{(i)} \partial x^{(l)}}\right)\sigma^{n,(i,j)}( x_{\kappa(n,s)}^n)\sigma^{n,(l,j)}( x_{\kappa(n,s)}^n)\,ds,\\	
I_8(t)& = \frac{1}{2}\sum_{i,l = 1}^d \sum_{j = 1}^m\int_{\kappa(n,t)}^t \frac{\partial^2 b^{(k)}(x_{\kappa(n,s)}^n)}{\partial x^{(i)} \partial x^{(l)}}\sigma^{n,(i,j)}( x_{\kappa(n,s)}^n)\sigma^{n,(l,j)}( x_{\kappa(n,s)}^n)\,ds,\\		
I_9(t)& = \frac{1}{2}\sum_{i,l = 1}^d \sum_{j = 1}^m\int_{\kappa(n,t)}^t \frac{\partial^2 b^{(k)}(x_s^n)}{\partial x^{(i)} \partial x^{(l)}}(\sigma^{n,(i,j)}( x_{\kappa(n,s)}^n)\sigma_M^{n,(l,j)}(s, x_{\kappa(n,s)}^n)\\
																			&\hspace{15em} +\sigma_M^{n,(i,j)}(s, x_{\kappa(n,s)}^n)\tilde{\sigma}^{n,(l,j)}(s, x_{\kappa(n,s)}^n))\,ds.
\end{align*}	
Note that 															
\begin{align*}
&\mathbb{E}|I_2(t)+I_8(t)- b_1^{n,(k)}(t, x_{\kappa(n,t)}^n)|^2\\
& \leq C\sum_{i,l = 1}^d \sum_{j = 1}^m\mathbb{E}\left|-\frac{n^{-3/2}|x_{\kappa(n,t)}^n|^{3\rho}}{(1+n^{-3/2}|x_{\kappa(n,t)}^n|^{3\rho})^2}%\right.\\
%&\hspace{10em} \left. \times
\int_{\kappa(n,t)}^t \frac{\partial^2 b^{(k)}(x_{\kappa(n,s)}^n)}{\partial x^{(i)} \partial x^{(l)}}\sigma^{(i,j)}( x_{\kappa(n,s)}^n)\sigma^{(l,j)}( x_{\kappa(n,s)}^n)\,ds\right|^2,
\end{align*}
which by applying Remark \ref{remark1} and Lemma \ref{lemmalp0} yields
\begin{align}\label{lemma9eq1}
\mathbb{E}|I_2(t)+I_8(t)- b_1^{n,(k)}(t, x_{\kappa(n,t)}^n)|^2&  \leq 	Cn^{-5}\mathbb{E}||x_{\kappa(n,t)}^n|^{3\rho}(1+|x_{\kappa(n,t)}^n|^{2\rho+1})|^2 \leq Cn^{-5},
\end{align}
for \(p_0 \geq 10\rho +2\). Moreover, notice that
\begin{equation}\label{lemma9eq2}
I_5(t)=b_2^{n,(k)}(t, x_{\kappa(n,t)}^n).
\end{equation}
Then, one obtains the following
\begin{align*}
&\mathbb{E}|b^{(k)}(x_t^n)-b^{(k)}(x_{\kappa(n,t)}^n)-b_1^{n,(k)}(t,x_{\kappa(n,t)}^n)-b_2^{n,(k)}(t,x_{\kappa(n,t)}^n)|^2\\
& \hspace{1em}\leq 2\mathbb{E}|I_1(t)+I_3(t)+I_4(t)+I_6(t)+I_7(t)+I_9(t)|^2+2\mathbb{E}|I_2(t)+I_8(t)-b_1^{n,(k)}(t, x_{\kappa(n,t)}^n)|^2\\
& \hspace{1em}\leq C(\mathbb{E}|I_1(t)|^2+\mathbb{E}|I_3(t)|^2+\mathbb{E}|I_4(t)|^2+\mathbb{E}|I_6(t)|^2+\mathbb{E}|I_7(t)|^2+\mathbb{E}|I_9(t)|^2)+Cn^{-5},
%& \leq \mathbb{E}|I_1(t)|^2+\mathbb{E}|J_3(t)|^2+\mathbb{E}|I_4(t)|^2+\mathbb{E}|I_6(t)|^2+\mathbb{E}|I_7(t)|^2+\mathbb{E}|I_9(t)|^2.
\end{align*}
for any \(t \in [0,T]\). To estimate \(\mathbb{E}|I_1(t)|^2\), applying Cauchy-Schwarz inequality and Remark \ref{remark1} yield
\begin{align*}	
\mathbb{E}|I_1(t)|^2\leq Cn^{-1}\int_{\kappa(n,t)}^t\mathbb{E}(1+|x_s^n|+|x_{\kappa(n,s)}^n|)^{2\rho-2}(1+|x_{\kappa(n,s)}^n|)^{2\rho+2}|x_s^n-x_{\kappa(n,s)}^n|^2\,ds,
\end{align*}
which further implies due to Young's inequality and H\"{o}lder's inequality
\begin{align*}	
\mathbb{E}|I_1(t)|^2 &\leq Cn^{-1}\int_{\kappa(n,t)}^t\left(\mathbb{E}(1+|x_s^n|^{p_0}+|x_{\kappa(n,s)}^n|^{p_0})\right)^{\frac{4\rho}{p_0}}\left(\mathbb{E}|x_s^n-x_{\kappa(n,s)}^n|^{\frac{2p_0}{p_0-4\rho}}\right)^{\frac{p_0-4\rho}{p_0}}\,ds,
\end{align*}
 for any \(t \in [0,T]\). By Lemma \ref{lemma7}, one obtains
 \[
 \mathbb{E}|I_1(t)|^2 \leq Cn^{-3},
 \]
 %%%%%%%%%%%%%%%%%%%%%%%%%%%%%%%%%%%%%%%%%%%%%%%%%%%%
 for any \(n \in \mathbb{N}\). As for $\mathbb{E}|I_3(t)|^2$, applying Cauchy-Schwarz inequality and Remark \ref{remark1} give
\begin{align*}	
\mathbb{E}|I_3(t)|^2 & \leq Cn^{-1}\int_{\kappa(n,t)}^t\mathbb{E}(1+|x_s^n|)^{2\rho}(|b_1^{n}(s, x_{\kappa(n,s)}^n)|^2+|b_2^{n}(s, x_{\kappa(n,s)}^n)|^2)\,ds,
\end{align*}
then one writes by using H\"{o}lder's inequality that
\begin{align*}	
\mathbb{E}|I_3(t)|^2 &\leq Cn^{-1}\int_{\kappa(n,t)}^t\left(\mathbb{E}(1+|x_s^n|^{p_0})\right)^{\frac{2\rho}{p_0}}(\mathbb{E}|b_1^{n}(s, x_{\kappa(n,s)}^n)|^{\frac{2p_0}{p_0-2\rho}}+\mathbb{E}|b_2^{n}(s, x_{\kappa(n,s)}^n)|^{\frac{2p_0}{p_0-2\rho}})^{\frac{p_0-2\rho}{p_0}}\,ds,
\end{align*}
for any \(t \in [0,T]\). Applying Lemma \ref{lemma6} yields
 \[
 \mathbb{E}|I_3(t)|^2 \leq Cn^{-3},
 \]
 %%%%%%%%%%%%%%%%%%%%%%%%%%%%%%%%%%%%%%%%%%%%%%%%%%%%
 for any \(n \in \mathbb{N}\). To estimate $\mathbb{E}|I_4(t)|^2$, one uses Cauchy-Schwarz inequality, Remark \ref{remark1} and Young's inequality to obtain
\begin{align*}	
\mathbb{E}|I_4(t)|^2 & \leq C\int_{\kappa(n,t)}^t\mathbb{E}(1+|x_s^n|+|x_{\kappa(n,s)}^n|)^{2\rho-2}(1+|x_{\kappa(n,s)}^n|)^{\rho+2}|x_s^n-x_{\kappa(n,s)}^n|^2\,ds\\
				& \leq C\int_{\kappa(n,t)}^t\mathbb{E}(1+|x_s^n|+|x_{\kappa(n,s)}^n|)^{3\rho}|x_s^n-x_{\kappa(n,s)}^n|^2\,ds,
\end{align*}
which implies due to Young's inequality and H\"{o}lder's inequality
\begin{align*}	
\mathbb{E}|I_4(t)|^2 &\leq C\int_{\kappa(n,t)}^t\left(\mathbb{E}(1+|x_s^n|^{p_0}+|x_{\kappa(n,s)}^n|^{p_0})\right)^{\frac{3\rho}{p_0}}\left(\mathbb{E}|x_s^n-x_{\kappa(n,s)}^n|^{\frac{2p_0}{p_0-3\rho}}\right)^{\frac{p_0-3\rho}{p_0}}\,ds,
\end{align*}
 for any \(t \in [0,T]\). One applies Lemma \ref{lemma7} to obtain
 \begin{equation}\label{i4}
 \mathbb{E}|I_4(t)|^2 \leq Cn^{-2},
 \end{equation}
 %%%%%%%%%%%%%%%%%%%%%%%%%%%%%%%%%%%%%%%%%%%%%%%%%%%%
 for any \(n \in \mathbb{N}\). As for $\mathbb{E}|I_6(t)|^2$, it can be written as
\begin{align*}	
\mathbb{E}|I_6(t)|^2 & \leq C\int_{\kappa(n,t)}^t\mathbb{E}(1+|x_s^n|)^{2\rho}|\sigma_M^{n}(s,x_{\kappa(n,s)}^n)|^2\,ds,
\end{align*}
which by using H\"{o}lder's inequality yields
\begin{align*}	
\mathbb{E}|I_6(t)|^2 &\leq C\int_{\kappa(n,t)}^t\left(\mathbb{E}(1+|x_s^n|^{p_0})\right)^{\frac{2\rho}{p_0}}\left(\mathbb{E}|\sigma_M^{n}(s,x_{\kappa(n,s)}^n)|^{\frac{2p_0}{p_0-2\rho}}\right)^{\frac{p_0-2\rho}{p_0}}\,ds,
\end{align*}
 for any \(t \in [0,T]\). By using Lemma \ref{lemmalp0} and Lemma \ref{lemma6}, one obtains
 \begin{equation}\label{i6}
 \mathbb{E}|I_6(t)|^2 \leq Cn^{-2},
 \end{equation}
 %%%%%%%%%%%%%%%%%%%%%%%%%%%%%%%%%%%%%%%%%%%%%%%%%%%%
 for any \(n \in \mathbb{N}\). In order to estimate $\mathbb{E}|I_7(t)|^2$, one uses Cauchy-Schwarz inequality and Remark \ref{remark1} to obtain
\begin{align*}	
\mathbb{E}|I_7(t)|^2 & \leq Cn^{-1}\int_{\kappa(n,t)}^t\mathbb{E}(1+|x_s^n|+|x_{\kappa(n,s)}^n|)^{2\rho-4}(1+|x_{\kappa(n,s)}^n|)^{2\rho+4}|x_s^n-x_{\kappa(n,s)}^n|^2\,ds\\
				& \leq Cn^{-1}\int_{\kappa(n,t)}^t\mathbb{E}(1+|x_s^n|+|x_{\kappa(n,s)}^n|)^{4\rho}|x_s^n-x_{\kappa(n,s)}^n|^2\,ds,
\end{align*}
which by applying Young's inequality and H\"{o}lder's inequality yields
\begin{align*}	
\mathbb{E}|I_7(t)|^2 &\leq Cn^{-1}\int_{\kappa(n,t)}^t\left(\mathbb{E}(1+|x_s^n|+|x_{\kappa(n,s)}^n|)^{p_0}\right)^{\frac{4\rho}{p_0}}\left(\mathbb{E}|x_s^n-x_{\kappa(n,s)}^n|^{\frac{2p_0}{p_0-4\rho}}\right)^{\frac{p_0-4\rho}{p_0}}\,ds,
\end{align*}
 for any \(t \in [0,T]\). Then applying Lemma \ref{lemmalp0} and Lemma \ref{lemma7}, one obtains
 \[
 \mathbb{E}|I_7(t)|^2 \leq Cn^{-3},
 \]
 %%%%%%%%%%%%%%%%%%%%%%%%%%%%%%%%%%%%%%%%%%%%%%%%%%%%
 for any \(n \in \mathbb{N}\). Finally for $\mathbb{E}|I_9(t)|^2 $, one writes
\begin{align*}	
\mathbb{E}|I_9(t)|^2 & \leq Cn^{-1}\int_{\kappa(n,t)}^t\mathbb{E}(1+|x_s^n|)^{2\rho-2}(1+|x_{\kappa(n,s)}^n|)^{\rho+2}|\sigma_M^{n}(s, x_{\kappa(n,s)}^n)|^2\,ds\\
				&\hspace{1em}+Cn^{-1}\int_{\kappa(n,t)}^t\mathbb{E}(1+|x_s^n|)^{2\rho-2}|\sigma_M^{n}(s, x_{\kappa(n,s)}^n)|^2|\tilde{\sigma}^{n}(s, x_{\kappa(n,s)}^n)|^2\,ds,
\end{align*}
which implies due to Young's inequality and H\"{o}lder's inequality
\begin{align*}	
\mathbb{E}|I_9(t)|^2 &\leq Cn^{-1}\int_{\kappa(n,t)}^t\left(\mathbb{E}(1+|x_s^n|+|x_{\kappa(n,s)}^n|)^{p_0}\right)^{\frac{3\rho}{p_0}}\left(\mathbb{E}|\sigma_M^{n}(s, x_{\kappa(n,s)}^n|^{\frac{2p_0}{p_0-3\rho}}\right)^{\frac{p_0-3\rho}{p_0}}\,ds\\
				&\hspace{1em}+Cn^{-1}\int_{\kappa(n,t)}^t\left(\mathbb{E}(1+|x_s^n|)^{p_0}\right)^{\frac{2\rho-2}{p_0}} \left(\mathbb{E}|\tilde{\sigma}^{n}(s, x_{\kappa(n,s)}^n)|^{\frac{2p_0}{p_0-2\rho+2}}|\sigma_M^{n}(s, x_{\kappa(n,s)}^n|^{\frac{2p_0}{p_0-2\rho+2}}\right)^{\frac{p_0-2\rho+2}{p_0}}\,ds
\end{align*}
 for any \(t \in [0,T]\). One can then apply Lemma \ref{lemmalp0} and Lemma \ref{lemma6} to the first term, and apply Cauchy-Schwarz inequality to the second term to obtain
 \begin{align*}	
\mathbb{E}|I_9(t)|^2 &\leq Cn^{-3}+Cn^{-1}\int_{\kappa(n,t)}^t\left(\mathbb{E}|\tilde{\sigma}^{n}(s, x_{\kappa(n,s)}^n)|^{\frac{4p_0}{p_0-2\rho+2}}\mathbb{E}|\sigma_M^{n}(s, x_{\kappa(n,s)}^n|^{\frac{4p_0}{p_0-2\rho+2}}\right)^{\frac{p_0-2\rho+2}{2p_0}}\,ds
\end{align*}
which, by using Lemma \ref{lemma6}, implies
 \[
 \mathbb{E}|I_9(t)|^2 \leq Cn^{-3},
 \]
 %%%%%%%%%%%%%%%%%%%%%%%%%%%%%%%%%%%%%%%%%%%%%%%%%%%%
 for any \(n \in \mathbb{N}\) and \(t \in [0,T]\). Therefore,
 \[
 \sup_{0\leq t \leq T}\mathbb{E}|b(x_t^n)-b(x_{\kappa(n,t)}^n)-b_1^n(t,x_{\kappa(n,t)}^n)-b_2^n(t,x_{\kappa(n,t)}^n)|^2 \leq Cn^{-2}+Cn^{-5} \leq Cn^{-2},
 \]
 for any \(n \in \mathbb{N}\), and the proof is complete.				
\end{proof}
Denote by \(e_t^n := x_t - x_t^n\) for any \(t \in [0,T]\), and define the stopping times as follows: for $R>0$,
\begin{equation}\label{stoppingtimedef}.
\tau_R:=\inf\{t \geq 0:|x_t| \geq R\}, \quad \tau_{n,R}': = \inf\{t \geq 0:|x_t^n| \geq R\}, \quad \nu_{n,R}: = \tau_R \wedge \tau_{n,R}'.
\end{equation}
\begin{lemma}\label{stoppingtime}
Assume \ref{a1} to \ref{a5} hold and $p_0 \geq 2(5\rho +1)$. Then, there exists a constant $C>0$ such that for any $s \in [0,T]$, the following inequality holds
\[
\mathbb{P}( s>\nu_{n,R} ) \leq CR^{-2},
\]
where $\nu_{n,R}$ is the stopping time defined in \eqref{stoppingtimedef}.
\end{lemma}
\begin{proof} By applying Markov inequality, one obtains
\begin{align*}
\mathbb{P}( s>\nu_{n,R} )& \leq \mathbb{P}\left( \sup_{u \leq s}|x_u|>R \right) + \mathbb{P}\left( \sup_{u \leq s}|x_u^n|>R \right)\\
								& \leq R^{-2}\mathbb{E}\left(\sup_{u\leq s}|x_u|^2\right)+R^{-2}\mathbb{E}\left(\sup_{u\leq s}|x_u^n|^2\right)\\
								&\leq CR^{-2}.
\end{align*}
Note that the last inequality holds since by Lemma \ref{classicallp0} and Lemma \ref{lemmalp0}, we have shown that the $p_0$-th moment of $x_t$ and $x_t^n$ are bounded uniformly in time, i.e. $\sup_{0\leq t \leq T}\mathbb{E}|x_t|^{p_0}\leq C$ and $\sup_{0\leq t\leq T}\mathbb{E}|x_t^n|^{p_0} \leq C$ for all $n \in \mathbb{N}$ and $p_0 \geq 4$. Then, one can obtain the uniform $\mathcal{L}^2$ bound  by using Lemma 5 in \cite{SabanisAoAP}, which originally appeared in \cite{GyongyKrylov}. %one can obtain the uniform $\mathcal{L}^2$ bound, i.e. 
\end{proof}
\begin{lemma}\label{lemma11}
Assume \ref{a1} to \ref{a5} hold and \(p_0 \geq 2(5\rho +1)\). Then, there exists a constant $C>0$, which is independent of $R$, such that for any \(n \in \mathbb{N}\) and \(t \in [0,T]\),
\begin{align*}
& \mathbb{E}\int_0^{t\wedge \nu_{n,R}} e_s^n(b(x_s^n)-b(x_{\kappa(n,s)}^n)-b_1^n(s,x_{\kappa(n,s)}^n)-b_2^n(s,x_{\kappa(n,s)}^n))\,ds\\
&\hspace{10em}\leq C\int_0^{t}\sup_{0\leq r\leq s}\mathbb{E}|e_{r \wedge \nu_{n,R}}^n|^2\,ds+Cn^{-\frac{5+\beta}{2}}+CR^{-\frac{2}{5}}n^{-2},
\end{align*}
where $\nu_{n,R}$ is the stopping time defined in \eqref{stoppingtimedef}.
\end{lemma}
\begin{proof}
First, for any $k = 1,\dots,d$, applying It\^o's formula to \(b^{(k)}(x_t^n)-b^{(k)}(x_{\kappa(n,t)}^n)\) gives \eqref{itob}. Then, by \eqref{lemma9eq1} and \eqref{lemma9eq2}, one obtains
\[
\mathbb{E}\int_0^{t\wedge \nu_{n,R}}e_s^{n,(k)}(b^{(k)}(x_s^n)-b^{(k)}(x_{\kappa(n,s)}^n)-b_1^{n,(k)}(s,x_{\kappa(n,s)}^n)-b_2^{n,(k)}(s,x_{\kappa(n,s)}^n))\,ds \leq \sum_{i=1}^7T_i(t) +T_8,
\]
where
\begin{align*}	
T_1(t)&= \mathbb{E}\int_0^{t\wedge \nu_{n,R}}e_s^{n,(k)}\sum_{i = 1}^d \int_{\kappa(n,s)}^s\left( \frac{\partial b^{(k)}(x_r^n)}{\partial x^{(i)}}-\frac{\partial b^{(k)}(x_{\kappa(n,r)}^n)}{\partial x^{(i)}}\right)b^{n,(i)}( x_{\kappa(n,r)}^n)\,dr\,ds,\\
T_2(t)&=\mathbb{E}\int_0^{t\wedge \nu_{n,R}}e_s^{n,(k)} \sum_{i = 1}^d \int_{\kappa(n,s)}^s \frac{\partial b^{(k)}(x_r^n)}{\partial x^{(i)}}(b_1^{n,(i)}(r, x_{\kappa(n,r)}^n)+b_2^{n,(i)}(r, x_{\kappa(n,r)}^n))\,dr\,ds,\\
T_3(t)&= \mathbb{E}\int_0^{t\wedge \nu_{n,R}}e_s^{n,(k)} \sum_{i = 1}^d \sum_{j = 1}^m \int_{\kappa(n,s)}^s \left( \frac{\partial b^{(k)}(x_r^n)}{\partial x^{(i)}}-\frac{\partial b^{(k)}(x_{\kappa(n,r)}^n)}{\partial x^{(i)}} \right) \sigma^{n,(i,j)}(x_{\kappa(n,r)}^n)\,dw_r^j\,ds,\\	
T_4(t)&= \mathbb{E}\int_0^{t\wedge \nu_{n,R}}e_s^{n,(k)}\sum_{i = 1}^d \sum_{j = 1}^m \int_{\kappa(n,s)}^s\frac{\partial b^{(k)}(x_r^n)}{\partial x^{(i)}}\sigma_M^{n,(i,j)}(r,x_{\kappa(n,r)}^n)\,dw_r^j	\,ds,\\
T_5(t)&=\frac{1}{2} \mathbb{E}\int_0^{t\wedge \nu_{n,R}}e_s^{n,(k)}\sum_{i,l = 1}^d \sum_{j = 1}^m\int_{\kappa(n,s)}^s \left(\frac{\partial^2 b^{(k)}(x_r^n)}{\partial x^{(i)} \partial x^{(l)}}-\frac{\partial^2 b^{(k)}(x_{\kappa(n,r)}^n)}{\partial x^{(i)} \partial x^{(l)}}\right)\\
																			&\hspace{15em} \times \sigma^{n,(i,j)}( x_{\kappa(n,r)}^n)\sigma^{n,(l,j)}( x_{\kappa(n,r)}^n)\,dr\,ds,\\		
T_6(t)&= \frac{1}{2}\mathbb{E}\int_0^{t\wedge \nu_{n,R}}e_s^{n,(k)}\sum_{i,l = 1}^d \sum_{j = 1}^m\int_{\kappa(n,s)}^s \frac{\partial^2 b^{(k)}(x_r^n)}{\partial x^{(i)} \partial x^{(l)}}(\sigma^{n,(i,j)}( x_{\kappa(n,r)}^n)\sigma_M^{n,(l,j)}(r, x_{\kappa(n,r)}^n)\\
																			&\hspace{15em} +\sigma_M^{n,(i,j)}(r, x_{\kappa(n,r)}^n)\tilde{\sigma}^{n,(l,j)}(r, x_{\kappa(n,r)}^n))\,dr\,ds,\\\
T_7(t)&= C\int_0^{t}\sup_{0\leq r\leq s}\mathbb{E}|e_{r\wedge \nu_{n,R}}^n|^2\,ds,\\
T_8(t)&= Cn^{-5}.
\end{align*}
%where \(T_7(t) = C\int_0^{t}\sup_{0\leq r\leq s}\mathbb{E}|e_{r\wedge \nu_{n,R}}^n|^2\,ds\) and \(T_8=Cn^{-5}\). 
To estimate $T_1(t)$, one applies Young's inequality and Remark \ref{remark1} to obtain
\begin{align*}	
T_1(t) & \leq C\int_0^{t}\sup_{0\leq r\leq s}\mathbb{E}|e_{r\wedge \nu_{n,R}}^n|^2\,ds\\
		&\hspace{1em}+ Cn^{-1}\int_0^{t}\int_{\kappa(n,s)}^s\mathbb{E}(1+|x_r^n|+|x_{\kappa(n,r)}^n|)^{4\rho}|x_r^n-x_{\kappa(n,r)}^n|^2\,dr\,ds,
\end{align*}
which by using H\"{o}lder's inequality implies
\begin{align*}	
T_1(t) &\leq C\int_0^{t}\sup_{0\leq r\leq s}\mathbb{E}|e_{r\wedge \nu_{n,R}}^n|^2\,ds\\
		&\hspace{1em}+ Cn^{-1}\int_0^{t}\int_{\kappa(n,s)}^s\left(\mathbb{E}(1+|x_r^n|^{p_0}+|x_{\kappa(n,r)}^n|^{p_0})\right)^{\frac{4\rho}{p_0}}\left(\mathbb{E}|x_r^n-x_{\kappa(n,r)}^n|^{\frac{2p_0}{p_0-4\rho}}\right)^{\frac{p_0-4\rho}{p_0}}\,dr\,ds.
\end{align*}
Thus, by Lemma \ref{lemma7} and Lemma \ref{lemmalp0}, one obtains
\[
T_1(t) 	\leq C\int_0^{t}\sup_{0\leq r\leq s}\mathbb{E}|e_{r\wedge \nu_{n,R}}^n|^2\,ds +Cn^{-3},
\]
for any \(n \in \mathbb{N}\). For $T_2(t)$, $T_5(t)$ and $T_6(t)$, the same results can be obtained by the direct application of Cauchy-Schwarz inequality combining with previous Lemmas and Remarks. The rest of the proof will mainly focus on obtaining estimates for $T_3(t)$ and $T_4(t)$. %To estimate $T_2(t)$, one writes
%\begin{align*}	
%T_2(t) 	&:= \mathbb{E}\int_0^{t\wedge \nu_{n,R}}e_s^n\sum_{i = 1}^d \int_{\kappa(n,s)}^s \frac{\partial b^{k}(x_r^n)}{\partial x^i}(b_1^{n,i}(r, x_{\kappa(n,r)}^n)+b_2^{n,i}(r, x_{\kappa(n,r)}^n))\,dr\,ds\\
%		& \leq C\int_0^{t\wedge \nu_{n,R}}\sup_{0\leq r\leq s}\mathbb{E}|e_r^n|^2\,ds+Cn^{-1}\int_0^{t\wedge \nu_{n,R}}\int_{\kappa(n,s)}^s\mathbb{E}\,dr\,ds
%\end{align*}
%\begin{align*}	
%T_5(t) 	&:= \mathbb{E}\int_0^{t\wedge \nu_{n,R}}e_s^n\left(\frac{1}{2}\sum_{i,l = 1}^d \sum_{j = 1}^m\int_{\kappa(n,s)}^s \left(\frac{\partial^2 b^{k}(x_r^n)}{\partial x^i \partial x^l}-\frac{\partial^2 b^{k}(x_{\kappa(n,r)}^n)}{\partial x^i \partial x^l}\right)\right.\\
%																			&\hspace{15em} \left.\times \sigma^{n,(i,j)}( x_{\kappa(n,r)}^n)\sigma^{n,(l,j)}( x_{\kappa(n,r)}^n)\,dr\right)\,ds\\
%		& \leq C\int_0^{t\wedge \nu_{n,R}}\sup_{0\leq r\leq s}\mathbb{E}|e_r^n|^2\,ds+Cn^{-1}\int_0^{t\wedge \nu_{n,R}}\int_{\kappa(n,s)}^s\mathbb{E}\,dr\,ds
%\end{align*}
%\begin{align*}	
%T_6(t) 	&:= \mathbb{E}\int_0^{t\wedge \nu_{n,R}}e_s^n\left(\frac{1}{2}\sum_{i,l = 1}^d \sum_{j = 1}^m\int_{\kappa(n,s)}^s \frac{\partial^2 b^{k}(x_r^n)}{\partial x^i \partial x^l}(\sigma^{n,(i,j)}( x_{\kappa(n,r)}^n)\sigma_M^{n,(l,j)}(r, x_{\kappa(n,r)}^n)\right.\\
%																			&\hspace{15em} +\left.\sigma_M^{n,(i,j)}(r, x_{\kappa(n,r)}^n)\tilde{\sigma}^{n,(l,j)}(r, x_{\kappa(n,r)}^n))\,dr\right)\,ds\\
%		& \leq C\int_0^{t\wedge \nu_{n,R}}\sup_{0\leq r\leq s}\mathbb{E}|e_r^n|^2\,ds+Cn^{-1}\int_0^{t\wedge \nu_{n,R}}\int_{\kappa(n,s)}^s\mathbb{E}\,dr\,ds
%\end{align*}
For any $r \in [0,T]$, $i,k = 1, \dots, d$ and $j = 1, \dots, m$, denote by%consider the following
\[
\mathbb{T}_r^{(i,j,k)} := \left( \frac{\partial b^{(k)}(x_r^n)}{\partial x^{(i)}}-\frac{\partial b^{(k)}(x_{\kappa(n,r)}^n)}{\partial x^{(i)}} \right) \sigma^{n,(i,j)}(x_{\kappa(n,r)}^n)+\frac{\partial b^{(k)}(x_r^n)}{\partial x^{(i)}}\sigma_M^{n,(i,j)}(r,x_{\kappa(n,r)}^n).
\]
Then, applying Remark \ref{remark1} and H\"{o}lder's inequality yields
\begin{align*}	
\mathbb{E}|\mathbb{T}_r^{(i,j,k)}|^p 	&=\mathbb{E}\left|\left( \frac{\partial b^{(k)}(x_r^n)}{\partial x^{(i)}}-\frac{\partial b^{(k)}(x_{\kappa(n,r)}^n)}{\partial x^{(i)}} \right) \sigma^{n,(i,j)}(x_{\kappa(n,r)}^n)+\frac{\partial b^{(k)}(x_r^n)}{\partial x^{(i)}}\sigma_M^{n,(i,j)}(r,x_{\kappa(n,r)}^n)\right|^p\\
								&\leq C\left(\mathbb{E}(1+|x_r^n|+|x_{\kappa(n,r)}^n|)^{p_0}\right)^{\frac{3\rho p}{2p_0}}\left(\mathbb{E}|x_r^n-x_{\kappa(n,r)}^n|^{\frac{2pp_0}{2p_0-3\rho p}}\right)^{\frac{2p_0-3\rho p}{2p_0}}\\
								&\hspace{10em}+C\left(\mathbb{E}(1+|x_r^n|)^{p_0}\right)^{\frac{\rho p}{p_0}}(\mathbb{E}|\sigma_M^{n,(i,j)}(r,x_{\kappa(n,r)}^n)|^{\frac{pp_0}{p_0-\rho p}})^{\frac{p_0-\rho p}{p_0}},
\end{align*}
which, by using Lemma \ref{lemma7} and Lemma \ref{lemma6}, implies
\begin{equation}\label{t2}
\sup_{r \leq T}\mathbb{E}|\mathbb{T}_r^{(i,j,k)}|^p  \leq Cn^{-\frac{p}{2}},
\end{equation}
%Similarly, it can be obtained
%\begin{equation}\label{t3}
%\mathbb{E}|\mathbb{T}|^4  \leq Cn^{-2},
%\end{equation}
%for \(p_0 \geq 16\rho+8\) and for all \(n \in \mathbb{N}\).
for $p \leq \frac{2p_0}{7\rho+2}$. Due to \eqref{i4} and \eqref{i6} in the proof of Lemma \ref{lemma10}, one can also obtain the following estimate
\begin{equation}\label{t1}	
\mathbb{E}\left|\sum_{i = 1}^d \sum_{j = 1}^m \int_{\kappa(n,s)}^s \mathbb{T}_r^{(i,j,k)} \,dw_r^j \right|^2 \leq 2\mathbb{E}|I_4(t)|^2+2\mathbb{E}|I_6(t)|^2\leq Cn^{-2}.
\end{equation}
Then, one writes
\begin{align*}	
T_3(t)+T_4(t): =\,& \mathbb{E}\int_0^{t\wedge \nu_{n,R}}e_s^{n,(k)} \sum_{i = 1}^d \sum_{j = 1}^m \int_{\kappa(n,s)}^s \mathbb{T}_r^{(i,j,k)} \,dw_r^j \,ds\\
				=\,& \mathbb{E}\int_0^{t\wedge \nu_{n,R}}(e_s^{n,(k)}-e_{\kappa(n,s)}^{n,(k)}) \sum_{i = 1}^d \sum_{j = 1}^m \int_{\kappa(n,s)}^s \mathbb{T}_r^{(i,j,k)} \,dw_r^j \,ds\\
					&+\mathbb{E}\int_0^{t\wedge \nu_{n,R}}e_{\kappa(n,s)}^{n,(k)} \sum_{i = 1}^d \sum_{j = 1}^m \int_{\kappa(n,s)}^s \mathbb{T}_r^{(i,j,k)} \,dw_r^j \,ds.
%				=\,& \mathbb{E}\int_0^{t\wedge \nu_{n,R}}(e_s^n-e_{\kappa(n,s)}^n) \sum_{i = 1}^d \sum_{j = 1}^m \int_{\kappa(n,s)}^s \mathbb{T}_r^{(i,j,k)} \,dw_r^j \,ds\\
%					&+\mathbb{E}\int_0^{t}\sum_{i = 1}^d \sum_{j = 1}^m \int_{\kappa(n,s)}^{(s\wedge \nu_{n,R})\vee \kappa(n,s)} e_{\kappa(n,s)}^n \mathbb{T}_r^{(i,j,k)} \,dw_r^j \,ds,
\end{align*}
Note that the second term above is not zero. However, by using Lemma \ref{stoppingtime}, one obtains%Since the second term above is zero, it can be further expressed as
\begin{align*}
&\mathbb{E}\int_0^{t\wedge \nu_{n,R}}e_{\kappa(n,s)}^{n,(k)} \sum_{i = 1}^d \sum_{j = 1}^m \int_{\kappa(n,s)}^s \mathbb{T}_r^{(i,j,k)} \,dw_r^j \,ds\\	
	&\hspace{1em} =\mathbb{E}\int_0^t \mathbf{1}_{\{s \leq \nu_{n,R}\}}e_{\kappa(n,s)\wedge \nu_{n,R}}^{n,(k) } \sum_{i = 1}^d \sum_{j = 1}^m \int_{\kappa(n,s)}^s \mathbb{T}_r^{(i,j,k)} \,dw_r^j \,ds\\
	&\hspace{1em} =\mathbb{E}\int_0^t e_{\kappa(n,s)\wedge \nu_{n,R}}^{n,(k) } \sum_{i = 1}^d \sum_{j = 1}^m \int_{\kappa(n,s)}^s \mathbb{T}_r^{(i,j,k)} \,dw_r^j \,ds\\
	&\hspace{4em}-\mathbb{E}\int_0^t \mathbf{1}_{\{s > \nu_{n,R}\}}e_{\kappa(n,s)\wedge \nu_{n,R}}^{n,(k) } \sum_{i = 1}^d \sum_{j = 1}^m \int_{\kappa(n,s)}^s \mathbb{T}_r^{(i,j,k)} \,dw_r^j \,ds,
\end{align*}
where the first term is zero since $\kappa(n,s)\wedge \nu_{n,R}$ is $\mathcal{F}_{\kappa(n,s)}$-measurable. Then, applying Young's inequality, H\"older's inequality to the second term yield
\begin{align*}
&\mathbb{E}\int_0^{t\wedge \nu_{n,R}}e_{\kappa(n,s)}^{n,(k)} \sum_{i = 1}^d \sum_{j = 1}^m \int_{\kappa(n,s)}^s \mathbb{T}_r^{(i,j,k)} \,dw_r^j \,ds\\	
	&\hspace{1em} \leq C\int_0^t \sup_{0\leq r\leq s}\mathbb{E}|e_{r \wedge \nu_{n,R}}^n|^2\,ds+C\int_0^t\left(\mathbb{P}\left(s > \nu_{n,R}\right)\right)^{\frac{1}{5}}\left(\mathbb{E}\left| \sum_{i = 1}^d \sum_{j = 1}^m \int_{\kappa(n,s)}^s \mathbb{T}_r^{(i,j,k)} \,dw_r^j\right|^{\frac{5}{2}}\right)^{\frac{4}{5}}\,ds\\
	&\hspace{1em}\leq  C\int_0^t \sup_{0\leq r\leq s}\mathbb{E}|e_{r \wedge \nu_{n,R}}^n|^2\,ds+CR^{-\frac{2}{5}}n^{-2},
\end{align*}
where the last inequality holds due to \eqref{t2}. One may notice that \eqref{t2} holds only when $\frac{5}{2} \leq \frac{2p_0}{7\rho+2}$, which implies $p_0 \geq \frac{5}{4}(7\rho+2)$. However, as $\frac{5}{4}(7\rho+2) \leq 2(5\rho+1)$ for all $\rho \geq 2$, by assuming $p_0 \geq 2(5\rho+1)$, \eqref{t2} holds automatically for $p= \frac{5}{2}$. Furthermore, $T_3(t)+T_4(t)$ can be expressed as
\begin{align*}	
T_3(t)+T_4(t)	=\,& \mathbb{E}\int_0^{t\wedge \nu_{n,R}}\int_{\kappa(n,s)}^s\bar{b}^{n,(k)}(r, x_{\kappa(n,r)}^n)\,dr \sum_{i = 1}^d \sum_{j = 1}^m \int_{\kappa(n,s)}^s \mathbb{T}_r^{(i,j,k)} \,dw_r^j \,ds\\
					&+\mathbb{E}\int_0^{t\wedge \nu_{n,R}}\sum_{v=1}^m\int_{\kappa(n,s)}^s\bar{\sigma}^{n,(k,v)}(r, x_{\kappa(n,r)}^n)\,dw_r^v \sum_{i = 1}^d \sum_{j = 1}^m \int_{\kappa(n,s)}^s \mathbb{T}_r^{(i,j,k)} \,dw_r^j \,ds\\
					&+  C\int_0^t \sup_{0\leq r\leq s}\mathbb{E}|e_{r \wedge \nu_{n,R}}^n|^2\,ds+CR^{-\frac{2}{5}}n^{-2},
\end{align*}
where \(\bar{b}^{n,(k)}(t,x_{\kappa(n,t)}^n)= b^{(k)}(x_t)- \tilde{b}^{n,(k)}(t,x_{\kappa(n,t)}^n)\) and \(\bar{\sigma}^{n, (k,v)}(t,x_{\kappa(n,t)}^n)= \sigma^{(k,v)}(x_t)- \tilde{\sigma}^{n,(k,v)}(t,x_{\kappa(n,t)}^n)\). One observes that $T_3(t)+T_4(t)$ can be expanded as
\begin{align*}	
&T_3(t)+T_4(t)	\\
&\hspace{1em}= \mathbb{E}\int_0^{t\wedge \nu_{n,R}}\int_{\kappa(n,s)}^s(b^{(k)}(x_r)-b^{(k)}(x_r^n))\,dr \sum_{i = 1}^d \sum_{j = 1}^m \int_{\kappa(n,s)}^s \mathbb{T}_r^{(i,j,k)}\,dw_r^j \,ds\\
					&\hspace{1em}+ \mathbb{E}\int_0^{t\wedge \nu_{n,R}}\int_{\kappa(n,s)}^s(b^{(k)}(x_r^n)-b^{(k)}(x_{\kappa(n,r)}^n)-b_1^{n,(k)}(r, x_{\kappa(n,r)}^n)-b_2^{n,(k)}(r, x_{\kappa(n,r)}^n))\,dr \\
					& \hspace{10em} \times\sum_{i = 1}^d \sum_{j = 1}^m \int_{\kappa(n,s)}^s \mathbb{T}_r^{(i,j,k)} \,dw_r^j \,ds\\
					&\hspace{1em}+\mathbb{E}\int_0^{t\wedge \nu_{n,R}}\int_{\kappa(n,s)}^s(b^{(k)}(x_{\kappa(n,r)}^n)-b^{n,(k)}(x_{\kappa(n,r)}^n)\,dr \sum_{i = 1}^d \sum_{j = 1}^m \int_{\kappa(n,s)}^s \mathbb{T}_r^{(i,j,k)} \,dw_r^j \,ds\\
					&\hspace{1em}+\sum_{v=j=1}^m\sum_{i = 1}^d \mathbb{E}\int_0^{t\wedge \nu_{n,R}}\int_{\kappa(n,s)}^s(\sigma^{(k,v)}(x_r)-\sigma^{(k,v)}(x_r^n))\mathbb{T}_r^{(i,j,k)} \,dr \,ds\\
					&\hspace{1em}+ \sum_{v= j=1}^m\sum_{i = 1}^d\mathbb{E}\int_0^{t\wedge \nu_{n,R}}\int_{\kappa(n,s)}^s(\sigma^{(k,v)}(x_r^n)-\sigma^{(k,v)}(x_{\kappa(n,r)}^n)-\sigma_M^{n,(k,v)}(r, x_{\kappa(n,r)}^n))\mathbb{T}_r^{(i,j,k)} \,dr \,ds\\
					&\hspace{1em}+ \sum_{v= j=1}^m\sum_{i = 1}^d \mathbb{E}\int_0^{t\wedge \nu_{n,R}}\int_{\kappa(n,s)}^s(\sigma^{(k,v)}(x_{\kappa(n,r)}^n)-\sigma^{n,(k,v)}(x_{\kappa(n,r)}^n))\mathbb{T}_r^{(i,j,k)} \,dr \,ds\\
					&\hspace{1em}+  C\int_0^t \sup_{0\leq r\leq s}\mathbb{E}|e_{r \wedge \nu_{n,R}}^n|^2\,ds+CR^{-\frac{2}{5}}n^{-2},
\end{align*}
which implies due to Remark \ref{remark1}, Young's inequality and Cauchy-Schwarz inequality
\begin{align*}	
&T_3(t)+T_4(t)	\\
&\hspace{1em}\leq  C\mathbb{E}\int_0^{t\wedge \nu_{n,R}}\left(\int_{\kappa(n,s)}^s(1+|x_r|+|x_r^n|)^{2\rho}\,dr\int_{\kappa(n,s)}^s|e_r^n|^2\,dr\right)^{\frac{1}{2}}\left| \sum_{i = 1}^d \sum_{j = 1}^m \int_{\kappa(n,s)}^s \mathbb{T}_r^{(i,j,k)} \,dw_r^j\right| \,ds\\
					&\hspace{1em}+ C\int_0^{t}\left(\mathbb{E}\left|\int_{\kappa(n,s)}^s|b(x_r^n)-b(x_{\kappa(n,r)}^n)-b_1^n(r, x_{\kappa(n,r)}^n)-b_2^n(r, x_{\kappa(n,r)}^n)|\,dr\right|^2\right. \\
					& \hspace{14em} \times\left. \sum_{i = 1}^d \sum_{j = 1}^m\mathbb{E}\left|\int_{\kappa(n,s)}^s \mathbb{T}_r^{(i,j,k)} \,dw_r^j\right|^2\right)^{1/2} \,ds\\
					&\hspace{1em}+C\int_0^{t}n \times n^{-1}\mathbb{E}\int_{\kappa(n,s)}^s|b(x_{\kappa(n,r)}^n)-b^n(x_{\kappa(n,r)}^n|^2\,dr\,ds+Cn^{-1} \sum_{i = 1}^d \sum_{j = 1}^m\int_0^{t}\mathbb{E}\left| \int_{\kappa(n,s)}^s \mathbb{T}_r^{(i,j,k)} \,dw_r^j \,\right|^2ds\\
					&\hspace{1em}+C\sum_{i = 1}^d \sum_{j = 1}^m\mathbb{E}\int_0^{t\wedge \nu_{n,R}}\int_{\kappa(n,s)}^s(1+|x_r|+|x_r^n|)^{\frac{\rho}{2}}|e_r^n||\mathbb{T}_r^{(i,j,k)}|\,dr \,ds\\
					&\hspace{1em}+ C\sum_{i = 1}^d \sum_{j = 1}^m\int_0^{t}\int_{\kappa(n,s)}^s\sqrt{\mathbb{E}|\sigma(x_r^n)-\sigma(x_{\kappa(n,r)}^n)-\sigma_M^n(r, x_{\kappa(n,r)}^n)|^2\mathbb{E}|\mathbb{T}_r^{(i,j,k)}|^2} \,dr \,ds\\
					&\hspace{1em}+ C\sum_{i = 1}^d \sum_{j = 1}^m\int_0^{t}\int_{\kappa(n,s)}^s\sqrt{\mathbb{E}|\sigma(x_{\kappa(n,r)}^n)-\sigma^n(x_{\kappa(n,r)}^n)|^2\mathbb{E}|\mathbb{T}_r^{(i,j,k)}|^2} \,dr \,ds\\
					&\hspace{1em}+  C\int_0^t \sup_{0\leq r\leq s}\mathbb{E}|e_{r \wedge \nu_{n,R}}^n|^2\,ds+CR^{-\frac{2}{5}}n^{-2},
\end{align*}
for any \(t \in [0,T]\). Then, by \eqref{t1}, \eqref{t2}, Lemma \ref{lemma10}, Lemma \ref{lemma8}, Lemma \ref{lemma9}, H\"{o}lder's inequality and Cauchy-Schwarz inequality, one obtains
\begin{align*}	
T_3(t)+T_4(t)	\leq\, 	&Cn^{-1} \mathbb{E}\int_0^{t}\int_{\kappa(n,s)}^s(1+|x_r|+|x_r^n|)^{2\rho}\,dr\left| \sum_{i = 1}^d \sum_{j = 1}^m \int_{\kappa(n,s)}^s \mathbb{T}_r^{(i,j,k)}\,dw_r^j\right|^2 \,ds\\
						&+Cn^{-1}\sum_{i = 1}^d \sum_{j = 1}^m\int_0^{t}\int_{\kappa(n,s)}^s\left(\mathbb{E}(1+|x_r|+|x_r^n|)^{p_0}\right)^{\frac{\rho}{p_0}}(\mathbb{E}|\mathbb{T}_r^{(i,j,k)}|^{\frac{2p_0}{p_0-\rho}})^{\frac{p_0-\rho}{p_0}} \,dr \,ds\\
						&+C\mathbb{E}\int_0^{t}\sup_{0\leq r\leq s}\mathbb{E}|e_{r\wedge \nu_{n,R}}^n|^2\,ds+Cn^{-\frac{5+\beta}{2}}+ Cn^{-3}+CR^{-\frac{2}{5}}n^{-2},
\end{align*}
which, by applying H\"{o}lder's inequality, yields
\begin{align*}	
T_3(t)+T_4(t)	\leq \, 	&Cn^{-1} \sum_{i = 1}^d \sum_{j = 1}^m \int_0^{t}\left(n^{-\frac{p_0}{2\rho}+1}\mathbb{E}\int_{\kappa(n,s)}^s(1+|x_r|+|x_r^n|)^{p_0}\,dr\right)^{\frac{2\rho}{p_0}}\\
						& \hspace{5em} \times \left(n^{-\frac{p_0}{p_0-2\rho}+1}\mathbb{E}\int_{\kappa(n,s)}^s |\mathbb{T}_r^{(i,j,k)}|^{\frac{2p_0}{p_0-2\rho}} \,dr\right)^{\frac{p_0-2\rho}{p_0}}\,ds\\
						&+C\mathbb{E}\int_0^{t}\sup_{0\leq r\leq s}\mathbb{E}|e_{r\wedge \nu_{n,R}}^n|^2\,ds+Cn^{-\frac{5+\beta}{2}}+CR^{-\frac{2}{5}}n^{-2},
\end{align*}
for any \(t \in [0,T]\). %Then, applying It\^o's formula to $b(x_r)$, $b(x_r^n)$ yields
Thus, by using Lemma \ref{lemmalp0} and \eqref{t2}, one obtains
\[
T_3(t)+T_4(t)  \leq C\int_0^{t}\sup_{0\leq r\leq s}\mathbb{E}|e_{r\wedge \nu_{n,R}}^n|^2\,ds+Cn^{-\frac{5+\beta}{2}}+CR^{-\frac{2}{5}}n^{-2},
\]
for any \(n \in \mathbb{N}\). Finally, notice that
\begin{align*}
&\mathbb{E}\int_0^{t\wedge \nu_{n,R}} e_s^n(b(x_s^n)-b(x_{\kappa(n,s)}^n)-b_1^n(s,x_{\kappa(n,s)}^n)-b_2^n(s,x_{\kappa(n,s)}^n))\,ds\\
&\hspace{1em} = \sum_{k=1}^d\mathbb{E}\int_0^{t\wedge \nu_{n,R}} e_s^{n,(k)}(b(x_s^{n,(k)})-b(x_{\kappa(n,s)}^{n,(k)})-b_1^n(s,x_{\kappa(n,s)}^{n,(k)})-b_2^{n,(k)}(s,x_{\kappa(n,s)}^n))\,ds\\
&\hspace{1em} \leq C\int_0^{t}\sup_{0\leq r\leq s}\mathbb{E}|e_{r\wedge \nu_{n,R}}^n|^2\,ds+Cn^{-\frac{5+\beta}{2}}+CR^{-\frac{2}{5}}n^{-2},
\end{align*} 
and the proof is complete. % Note that in the last term, the constant $C>0$ depends on \(t\wedge \nu_{n,R}\). Finally, on the application of completes the proof.
\end{proof}
\textbf{Proof of Theorem \ref{theorem1}.} Applying It\^o's formula to $|e_{t\wedge \nu_{n,R}}^n|^2$ gives, almost surely,
\begin{align*}	
|e_{t\wedge \nu_{n,R}}^n|^2 	& = 2\int_0^{t\wedge \nu_{n,R}}e_s^n\bar{b}^{n}(s, x_{\kappa(n,s)}^n)\,ds+2\int_0^{t\wedge \nu_{n,R}}e_s^n\bar{\sigma}^{n}(s, x_{\kappa(n,s)}^n)\,dw_s + \int_0^{t\wedge \nu_{n,R}} |\bar{\sigma}^{n}(s, x_{\kappa(n,s)}^n)|^2\,ds,
\end{align*}
where $\nu_{n,R}$ is the stopping time defined in \eqref{stoppingtimedef}, \(\bar{b}^n(t,x_{\kappa(n,t)}^n)= b(x_t)- \tilde{b}^n(t,x_{\kappa(n,t)}^n)\) and \(\bar{\sigma}^{n}(t,x_{\kappa(n,t)}^n)= \sigma(x_t)- \tilde{\sigma}^n(t,x_{\kappa(n,t)}^n)\). Taking expectations on both sides and using Young's inequality yield, for any $\varepsilon >0$,
\begin{align*}	
\mathbb{E}|e_{t\wedge \nu_{n,R}}^n|^2 	& \leq 2\mathbb{E}\int_0^{t\wedge \nu_{n,R}}e_s^n (b(x_s)-b(x_s^n))\,ds\\
													&\hspace{1em} + 2\mathbb{E}\int_0^{t\wedge \nu_{n,R}}e_s^n (b(x_s^n)-b(x_{\kappa(n,s)}^n)-b^n_1(s, x_{\kappa(n,s)}^n)-b^n_2(s, x_{\kappa(n,s)}^n))\,ds\\
													&\hspace{1em} + 2\mathbb{E}\int_0^{t\wedge \nu_{n,R}}e_s^n (b(x_{\kappa(n,s)}^n)-b^n(x_{\kappa(n,s)}^n))\,ds + (1+\varepsilon)\mathbb{E}\int_0^{t\wedge \nu_{n,R}} |\sigma (x_s)-\sigma(x_s^n)|^2\,ds\\
													& \hspace{1em} + C\mathbb{E}\int_0^{t\wedge \nu_{n,R}} |\sigma(x_s^n)-\sigma(x_{\kappa(n,s)}^n)-\sigma^n_M(s, x_{\kappa(n,s)}^n)|^2\,ds\\
													& \hspace{1em} + C\mathbb{E}\int_0^{t\wedge \nu_{n,R}} |\sigma(x_{\kappa(n,s)}^n)-\sigma^n(x_{\kappa(n,s)}^n)|^2\,ds.
\end{align*}
for any $t \in [0,T]$. Then, by Cauchy-Schwarz inequality, one obtains
\begin{align*}	
\mathbb{E}|e_{t\wedge \nu_{n,R}}^n|^2 	& \leq  \mathbb{E}\int_0^{t\wedge \nu_{n,R}}(2e_s^n (b(x_s)-b(x_s^n))+ (1+\varepsilon)|\sigma(x_s)-\sigma(x_s^n)|^2)\,ds\\
													&\hspace{1em} + 2\mathbb{E}\int_0^{t\wedge \nu_{n,R}}e_s^n (b(x_s^n)-b(x_{\kappa(n,s)}^n)-b^n_1(s, x_{\kappa(n,s)}^n)-b^n_2(s, x_{\kappa(n,s)}^n))\,ds\\
													&\hspace{1em} + \mathbb{E}\int_0^{t\wedge \nu_{n,R}}|b(x_{\kappa(n,s)}^n)-b^n(x_{\kappa(n,s)}^n)|^2\,ds\\
													& \hspace{1em} + C\mathbb{E}\int_0^{t\wedge \nu_{n,R}} |\sigma(x_s^n)-\sigma(x_{\kappa(n,s)}^n)-\sigma^n_M(s, x_{\kappa(n,s)}^n)|^2\,ds\\
													& \hspace{1em} +C \mathbb{E}\int_0^{t\wedge \nu_{n,R}} |\sigma(x_{\kappa(n,s)}^n)-\sigma^n(x_{\kappa(n,s)}^n)|^2\,ds+\int_0^{t}\sup_{0\leq r\leq s}\mathbb{E}|e_{r\wedge \nu_{n,R}}^n|^2\,ds.
\end{align*}
Since \(p_1>2\), applying \ref{a3} to the first term, and applying Lemma \ref{lemma11}, \ref{lemma8} and \ref{lemma9} yield
\begin{align*}	
\sup_{0 \leq s \leq t}\mathbb{E}|e_{s\wedge \nu_{n,R}}^n|^2 	\leq C\int_0^{t}\sup_{0\leq r\leq s}\mathbb{E}|e_{r\wedge \nu_{n,R}}^n|^2\,ds +Cn^{-(2+\beta)}+CR^{-\frac{2}{5}}n^{-2}<\infty,
\end{align*}
for any $t \in [0,T]$ and \(n \in \mathbb{N}\). Finally, one applies Gronwall's lemma to obtain
\[
\sup_{0 \leq s \leq t}\mathbb{E}|e_{s\wedge \nu_{n,R}}^n|^2 	\leq Cn^{-(2+\beta)}+CR^{-\frac{2}{5}}n^{-2},
\]
and the proof is complete by using Fatou's lemma, since the last term in the above inequality vanishes as $R$ tends to infinity.
%%%%%%%%%%%%%%%%%%%%%%%%%%%%%%%%%%%%%%%%%%%%%%%%%%%%
\section{Simulation results}
In this section, simulation results are provided to support the theoretical results in the previous sections. Consider $T =1$, the step size $\Delta = t_{k+1} - t_k=1/N$ for $N \in \mathbb{N}$, $t_0=0$, and $k \in \{0,\dots,N-1\}$. For the case $d= m =1$, the discrete version of the order 1.5 scheme \eqref{scheme} is as follows: %Denote by $X_k =X_{t_k}$, for $k \in \{1, \dots, N\}$. 
\begin{align*}
X_{k+1} & = X_k + b^n \Delta + \sigma^n \Delta W +L^{n,1}b \Delta Z+ \frac{1}{2} L^{n,0}b \Delta ^2 \\
& \hspace{1em}+\frac{1}{2}L^{n,1}\sigma((\Delta W)^2-\Delta) + L^{n,0}\sigma (\Delta W \Delta - \Delta Z)\\
& \hspace{1em}+\frac{1}{2}L^{n,1}L^1 \sigma\left(\frac{1}{3}(\Delta W)^2 - \Delta\right) \Delta W,
\end{align*}
where the following conventions are used: $X_k =X_{t_k}$, \(\Delta W = W_{t_{k+1}} - W_{t_k}\) and %$b^n$ and $\sigma^n$ are the tamed drift and diffusion coefficients of the SDE, and $\Delta Z$ is defined as
$\Delta Z = \int_{t_k}^{t_{k+1}}\int_{t_k}^{s}\,dW_r\,ds$. Note that $\Delta Z$ is normally distributed with mean zero, variance $\frac{1}{3}\Delta^3$, and covariance 
\[
\mathbb{E}(\Delta Z \Delta W) = \frac{1}{2}\Delta^2.
\]
Then, the following two examples are considered. For the first example, the one-dimensional SDE is given by
\begin{equation}\label{lipschitzexample}
dx_t = x_t(1-x_t^2)dt + \xi (1-x_t^2)dw_t, \quad \forall t \in [0,T],
\end{equation}
where $T \geq 0$ and $\xi \in [-0.3086,0.3086]$. As for the second example, one consider the SDE % governed by
\begin{equation}\label{holderexample}
dx_t = x_t(1-|x_t|^3)dt + \xi|x_t|^\frac{5}{2}dw_t, \quad \forall t \in [0,T],
\end{equation}
where $T \geq 0$ and $\xi \in [-0.2209,0.2209]$. One can check (see Appendix) that the first example \eqref{lipschitzexample} satisfies the assumptions \ref{a1} to \ref{a5} with $\rho = 2$, whereas the second example \eqref{holderexample} satisfies the assumptions with $\rho = 4$.
\begin{figure}%[!htb]
\captionsetup{width=0.8\textwidth}
\centering
    \begin{subfigure}[t]{0.5\textwidth}
        \centering
        \includegraphics[scale=.55]{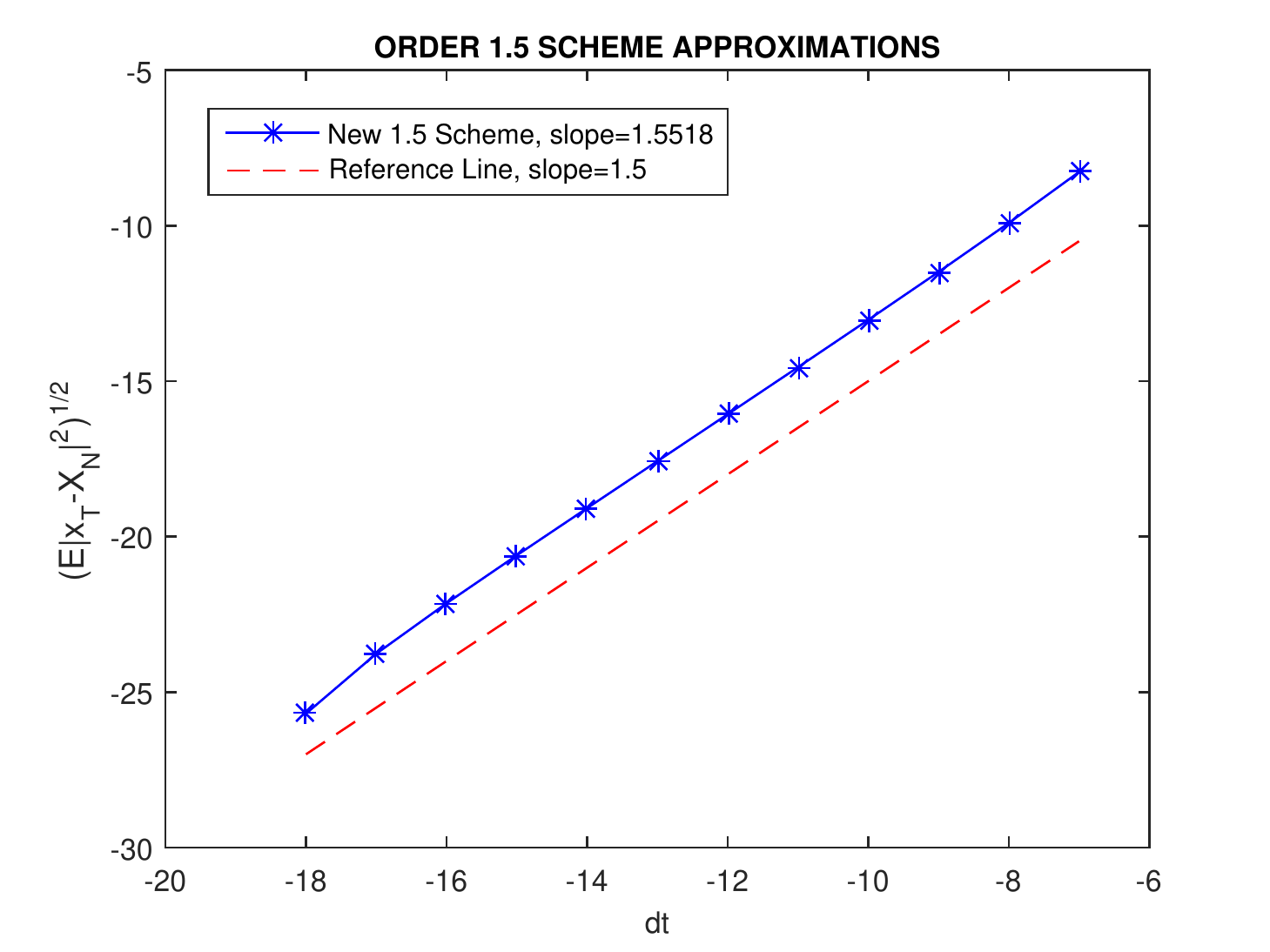}
        \caption{$\beta = 1$}
        \label{fig:beta1}
    \end{subfigure}%
    \begin{subfigure}[t]{0.5\textwidth}
        \centering
        \includegraphics[scale=.55]{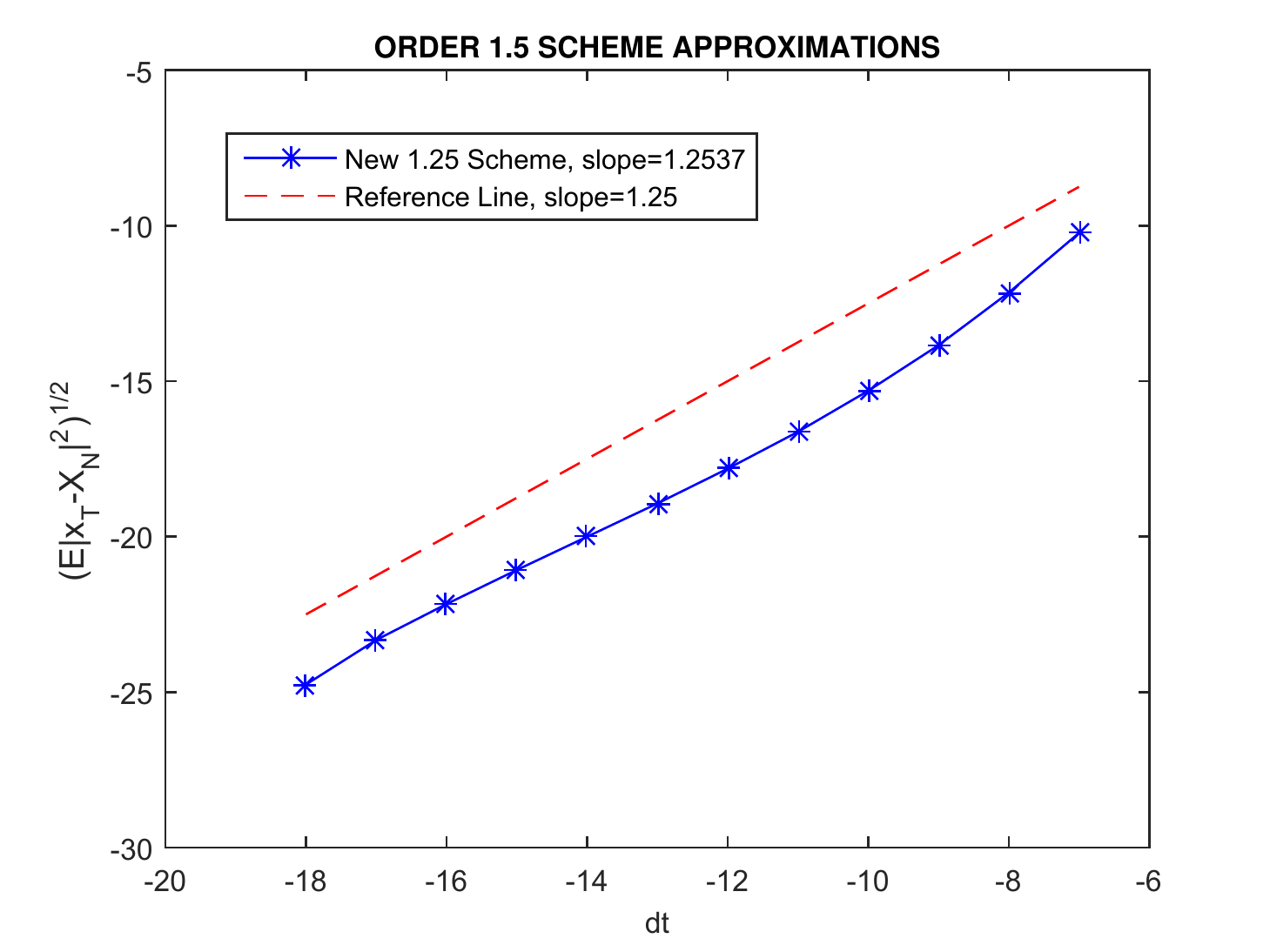}
        \caption{$\beta = 0.5$}
        \label{fig:beta0.5}
    \end{subfigure}%

    \caption{Rate of convergence of the new order 1.5 scheme with parameters $x_0 = 3$, $\xi = 0.02$ and $T=1$. Denote by $x_T$ and $X_N$ respectively the true solution and the numerical approximation of the corresponding SDE at time $T$. The dashed red lines are the reference lines, and the blue doted lines are the numerical results obtained using the scheme.}
\label{fig:rateofconvergence}
\end{figure}
As for the numerical results, Figure \ref{fig:rateofconvergence} above shows the rate of convergence of the scheme, and the approximations are obtained by simulating 1000 paths. Furthermore, Figure \ref{fig:beta1} illustrates that, for the case $\beta = 1$, the new explicit order 1.5 scheme has a rate of convergence estimate close to the theoretical result 1.5, which is 1.5518. Similarly, as shown in Figure \ref{fig:beta0.5},  the slope of the blue line is equal to 1.2537, which supports the theoretical prediction 1.25. Note that the examples considered in this section are one dimensional. However, in order to  implement such algorithms to real world problems where $d\geq 2$, the diffusion coefficient needs to satisfy the commutative condition. Otherwise, one needs to handle the associated Levy areas. One possible approach is to use a coupling technique (see \cite{coupling}).% However, notice that for $d$ (\(d \geq 2\)) dimensional case, in order to do the numerical implementation, the diffusion matrix $\sigma$ is assumed to satisfy the commutativity conditions
%\begin{equation}\label{eqn:commutativity1}
%L^{j}\sigma^{i,j_1} = L^{j_1}\sigma^{i,j}
%\end{equation}
%and
%\begin{equation}\label{eqn:commutativity2}
%L^{j}L^{j_1}\sigma^{i,j_2} = L^{j_1}L^j\sigma^{i,j_2},
%\end{equation}
%for all \(j, j_1, j_2 = 1, \dots, m\) and \(i = 1, \dots, d\). %, so that one can write explicitly the numerical scheme.

\section{Appendix}
\begin{enumerate}
\item Consider the one-dimensional SDE
\[
dx_t = x_t(1-x_t^2)dt + \xi (1-x_t^2)dw_t, \quad \forall t \in [0,T].
\]
\begin{enumerate}
\item \ref{a1} is satisfied as $x_0$ is taken to be a constant (i.e. $x_0=3$).
\item To verify \ref{a2}, one calculates
\begin{align*}
2xb(x) +(p_0 -1)|\sigma(x)|^2 	&=2x^2-2x^4 +(p_0-1)\xi^2(1-x^2)^2\\
													&=(p_0-1)\xi^2+2(1-\xi^2(p_0-1))x^2+(\xi^2(p_0-1)-2)x^4.
\end{align*}
We require $\xi^2(p_0-1)-2\leq 0$, which implies $p_0 \leq \frac{2}{\xi^2}+1$.%Since there exists $K>0$, such that $2xb(x) +(p_0 -1)|\sigma(x)|^2 \leq  K(1+|x|^2)$, we require $\xi^2 \in [0,\frac{2}{p_0-1}]$, and $K$ could be $2+2\xi^2(p_0-1)$.
\item As for \ref{a3}, one writes
\begin{align*}
&2(x-\bar{x})(b(x)-b(\bar{x})) +(p_1 -1)|\sigma(x)-\sigma (\bar{x})|^2 \\
&\hspace{1em}=2(x-\bar{x})((x-x^3)-(\bar{x}-\bar{x}^3))+(p_1 -1)\xi^2|(1-x^2)-(1-\bar{x}^2)|^2\\
%&\hspace{1em}=2(x-\bar{x})^2-2(x-\bar{x})(x^3-\bar{x}^3)+(p_1 -1)\xi^2|x+\bar{x}|^2|x-\bar{x}|^2\\
&\hspace{1em}=2(x-\bar{x})^2-2(x-\bar{x})^2((x+\bar{x})^2-x\bar{x})+(p_1 -1)\xi^2|x+\bar{x}|^2|x-\bar{x}|^2\\
&\hspace{1em} \leq 2(x-\bar{x})^2+(x-\bar{x})^2\left((p_1 -1)\xi^2|x+\bar{x}|^2-(x+\bar{x})^2\right).
\end{align*}
Then, in order to guarantee $2(x-\bar{x})(b(x)-b(\bar{x})) +(p_1 -1)|\sigma(x)-\sigma (\bar{x})|^2 \leq K|x-\bar{x}|^2$ is satisifed for some $K>0$, we require $p_1 \in (2,\frac{1}{\xi^2}+1]$.% In this case, $K=2$.
\item The second derivative of $b(x) = x(1-x^2)$ is $-6x$, then \ref{a4} is satisfied with $\rho \geq 2$ since
\[
\left|\frac{\partial^2 b(x)}{\partial x^2}-\frac{\partial^2 b(\bar{x})}{\partial \bar{x}^2}\right|\leq 6|x-\bar{x}|
\]
\item Similary, one can calculate the second derivative of $\sigma(x)= \xi (1-x^2)$, which is $-2\xi$. The assumption \ref{a5} is satisfied with $\rho \geq 2$.
\end{enumerate}
We choose $\rho$ to be 2, then, since it is assumed in Theorem 1 that $p_0 \geq 2(5\rho+1)=22$, one obtains $\xi \in [-0.3086,0.3086]$ by using $p_0 \in [22,\frac{2}{\xi^2}+1]$ and $p_1 \in (2,\frac{1}{\xi^2}+1]$.
\item As for the second example, consider the one-dimensional SDE
\[
dx_t = x_t(1-|x_t|^3)dt + \xi|x_t|^\frac{5}{2}dw_t, \quad \forall t \in [0,T].
\]
\begin{enumerate}
\item We take $x_0=3$, therefore \ref{a1} is satisfied.
\item As for \ref{a2}, one calculates
\begin{align*}
2xb(x) +(p_0 -1)|\sigma(x)|^2 	&=2x^2-2|x|^5 +(p_0-1)\xi^2|x|^5\\
													&=2x^2+((p_0-1)\xi^2-2)|x|^5.
\end{align*}
To guarantee \ref{a2} is satisfied, we require $p_0 \leq \frac{2}{\xi^2}+1$.
\item To verify \ref{a3}, one calculates the following
\begin{align*}
&2(x-\bar{x})(b(x)-b(\bar{x})) +(p_1 -1)|\sigma(x)-\sigma (\bar{x})|^2 \\
&\hspace{1em}=2(x-\bar{x})((x-x|x|^3)-(\bar{x}-\bar{x}|\bar{x}|^3))+(p_1 -1)\xi^2\left||x|^{\frac{5}{2}}-|\bar{x}|^{\frac{5}{2}}\right|^2\\
%&\hspace{1em}=2(x-\bar{x})^2-2(x-\bar{x})(x^3-\bar{x}^3)+(p_1 -1)\xi^2|x+\bar{x}|^2|x-\bar{x}|^2\\
&\hspace{1em}=2(x-\bar{x})^2-2(|x|^5-x\bar{x}|x|^3-x\bar{x}|\bar{x}|^3+|\bar{x}|^5)+(p_1 -1)\xi^2\left||x|^{\frac{5}{2}}-|\bar{x}|^{\frac{5}{2}}\right|^2\\
&\hspace{1em}\leq 2(x-\bar{x})^2+\left(-2|x|^5-2|\bar{x}|^5+\frac{6}{5}|x|^5+\frac{6}{5}|\bar{x}|^5+\frac{8}{5}|x|^{\frac{5}{2}}|\bar{x}|^{\frac{5}{2}}\right)+(p_1 -1)\xi^2\left||x|^{\frac{5}{2}}-|\bar{x}|^{\frac{5}{2}}\right|^2\\
&\hspace{1em} = 2(x-\bar{x})^2+\left((p_1 -1)\xi^2-\frac{4}{5}\right)\left||x|^{\frac{5}{2}}-|\bar{x}|^{\frac{5}{2}}\right|^2.
\end{align*}
Therefore, we require $p_1 \in (2,\frac{4}{5\xi^2}+1]$ for \ref{a3} to be satisfied.
\item The second derivative of $b(x) = x(1-|x|^3)$ is $-12x|x|$, then \ref{a4} is satisfied with $\rho \geq 3$ since
\begin{align*}
\left|\frac{\partial^2 b(x)}{\partial x^2}-\frac{\partial^2 b(\bar{x})}{\partial \bar{x}^2}\right|
&\leq 12|\bar{x}|\bar{x}|-x|x||\\
&=12|\bar{x}|\bar{x}|-x|\bar{x}|+x|\bar{x}|-x|x||\\
&\leq 12|\bar{x}||\bar{x}-x|+|x||\bar{x}-x|\\
& \leq 12(|x|+|\bar{x}|)|\bar{x}-x|\\
& \leq 12(1+|x|+|\bar{x}|)|\bar{x}-x|.
\end{align*}
\item The second derivative of $\sigma(x) = \xi |x|^{\frac{5}{2}}$ is $\frac{15}{4}\xi|x|^{\frac{1}{2}}$, then one obtains
\begin{align*}
\left|\frac{\partial^2 \sigma(x)}{\partial x^2}-\frac{\partial^2 \sigma(\bar{x})}{\partial \bar{x}^2}\right| &\leq\frac{15}{4}|\xi|\left||x|^{\frac{1}{2}}-|\bar{x}|^{\frac{1}{2}}\right| \leq \frac{15}{4}|\xi||x-\bar{x}|^{\frac{1}{2}},
\end{align*}
which implies that \ref{a5} is satisfied with $\rho \geq 4$, and the last inequality holds since
\[
\left||x|^{\frac{1}{2}}-|\bar{x}|^{\frac{1}{2}}\right|^2 \leq \left||x|^{\frac{1}{2}}-|\bar{x}|^{\frac{1}{2}}\right|\left||x|^{\frac{1}{2}}+|\bar{x}|^{\frac{1}{2}}\right| \leq ||x|-|\bar{x}|| \leq |x-\bar{x}|.
\]
\end{enumerate}
We choose $\rho =4$, then, as it is assumed in Theorem 1 that $p_0 \geq 2(5\rho+1)=42$, one obtains $\xi \in [-0.2209,0.2209]$ by using $p_0 \in [42,\frac{2}{\xi^2}+1]$ and $p_1 \in (2,\frac{4}{5\xi^2}+1]$.%and $\xi^2 \in [0,\frac{2}{41}]$, and we take $\xi=0.02$ which falls in this interval. Note that $p_1 \in (2,2001]$.
\end{enumerate}

\section*{Acknowledgements} 
 We  are grateful to anonymous referees for their helpful suggestions.

Ying Zhang was supported by The Maxwell Institute Graduate School in Analysis and its Applications, a Centre for Doctoral Training funded by the UK Engineering and Physical Sciences Research Council (grant EP/L016508/01), the Scottish Funding Council, Heriot-Watt University and the University of Edinburgh.

%-----------------------------------------------------------------

\end{document}